\documentclass[11 pt]{amsart}
\usepackage{amsmath,enumerate,amsthm,amssymb,amscd}

\usepackage{amsfonts,amstext}
\usepackage{graphicx}
\usepackage{amsmath,amstext,amssymb,amscd}
\usepackage{verbatim} 

\usepackage{amssymb,latexsym}

\usepackage{amsmath}
\usepackage{amsthm}
\usepackage{amssymb}
\newtheorem{thm}{Theorem}[section]
\newtheorem{cor}[thm]{Corollary}
\newtheorem{lem}[thm]{Lemma}
\newtheorem{prop}[thm]{Proposition}
\newtheorem*{thm1}{Theorem \ref{trace monoids aut sgps}}
\newtheorem*{thm2}{Theorem \ref{post corr thm}}
\newtheorem*{prop1}{Proposition}

\theoremstyle{remark}

\theoremstyle{definition}
\newtheorem{mydef}[thm]{Definition}
\newtheorem{ex}[thm]{Example}

\usepackage{amsthm}

\begin{document}
\title{Semigroups Arising from Asynchronous Automata}
\author{David McCune}
\address{Department of
Mathematics\\
University of Nebraska-Lincoln\\Lincoln,  NE 68588-0130}

\email{s-dmccune1@math.unl.edu}
\date{November 2010}
\begin{abstract}
We introduce a new class of semigroups arising from a restricted
class of asynchronous automata. We call these semigroups ``expanding
automaton semigroups." We show that the class of synchronous
automaton semigroups is strictly contained in the class of expanding
automaton semigroups, and that the class of expanding automaton
semigroups is strictly contained in the class of asynchronous
automaton semigroups.  We investigate the dynamics of expanding
automaton semigroups acting on regular rooted trees, and show that
undecidability arises in these actions.  We show that this class is
not closed under taking normal ideal extensions, but the class of
asynchronous automaton semigroups is closed under taking these
extensions.  We construct every free partially commutative monoid as
a synchronous automaton semigroup.
\end{abstract}

\bibliographystyle{amsplain}

\numberwithin{thm}{section}

\maketitle

\section{INTRODUCTION}

 Automaton groups were introduced in the 1980's as
 examples of groups with fascinating properties.  For example, Grigorchuk's
 group is the first known group of intermediate growth and is also
 an infinite periodic group.  Besides having interesting properties,
 many of these groups have deep connections with dynamical systems
 which were explored by Bartholdi and Nekrashevych in \cite{BN} and \cite{N}.  In
 particular, they use these groups to solve a
 longstanding problem in holomorphic dynamics (see \cite{BN}).  For a general
 introduction to these groups, see \cite{GS} by Grigorchuk and Sunic or \cite{N} by Nekrashevych.

 Many generalizations of automaton groups have been
 studied.  The most famous and well-studied generalization is the
 class self-similar groups.  A good introduction to these groups can
 be found in \cite{GS} or \cite{N}.
 More recently, Slupik and Sushchansky study semigroups arising from partial invertible synchronous
 automata in \cite{SlSu}.  Cain, Reznikov, Silva, Sushchanskii, and Steinberg investigate automaton
semigroups, which are semigroups that arise from (not necessarily
invertible) synchronous automata in \cite{C}, \cite{RezS}, and
\cite{SS}.  Grigorchuk, Nekrashevich, and Sushchanskii study groups
arising from asynchronous automata in \cite{GVS}.

 In all of the references listed above except for \cite{GVS}, the semigroups
 studied arose from synchronous automata.  In \cite{GVS}, Grigorchuk et al. study groups arising from asynchronous
 automata.  In particular, they give examples of automata generating
 Thompson groups and groups of shift automorphisms.  This paper studies a class of semigroups that we call ``expanding automaton
 semigroups."  These semigroups arise from a restricted class of asynchronous automata
 that we call ``expanding automata," and the class of expanding automata contains the class of
 synchronous automata.  Thus the class of automaton semigroups is contained in the class of expanding automaton semigroups, and
 the class of expanding automaton semigroups is contained in the class of asynchronous automaton semigroups.  As mentioned above, automaton semigroups and
 asynchronous automaton semigroups have been studied, but thus far a
 study of expanding automaton semigroups has not been done.

 In Section 2 we give definitions of the different kinds of automata, and
 explain how the states of a given automaton act on a regular
 rooted tree.  In particular, let $\Sigma$ be a finite set, and let $\Sigma^*$ denote the free monoid generated by
 $\Sigma$.  Then the states of a given automaton act on $\Sigma^*$
 for some finite set $\Sigma$.  Thus we can consider the semigroup generated by
 the states of an automaton as a semigroup of functions from $\Sigma^*$
 to $\Sigma^*$.  Given a free monoid $\Sigma^*$, we
 associate a regular rooted tree $\mathcal{T}(\Sigma^*)$ with $\Sigma^*$ by letting
 the vertices of $\mathcal{T}(\Sigma^*)$ be $\Sigma^*$ and
 letting the edge set be $(w,w\sigma)$ for all $w \in \Sigma^*$ and
 $\sigma \in \Sigma$.  The identity of $\Sigma^*$ is the root of the
 tree.
 The action of a semigroup associated with an asynchronous automaton
 on $\Sigma^*$ induces an action on
 the tree $\mathcal{T}(\Sigma^*)$.
 Let $\Sigma^{\omega}$ denote the set
 of right-infinite words over $\Sigma$.  Then $\Sigma^{\omega}$ is
 the boundary of the tree $\mathcal{T}(\Sigma^*)$.  The action of an
 asynchronous automaton semigroup on $\Sigma^*$ induces an action of the semigroup on
 $\Sigma^{\omega}$, and so an asynchronous automaton semigroup acts
 on the boundary of a regular rooted tree.

 Section 2 also contains examples of expanding automaton
 semigroups that are not automaton semigroups (Proposition \ref{not aut sgps}), as well as
 asynchronous automaton semigroups that are not expanding automaton
 semigroups (Proposition \ref{bicyclic submonoid}).  Thus Propositions \ref{not aut sgps} and \ref{bicyclic submonoid}
 combine to show the following.
\begin{prop1}
 The class of automaton semigroups is strictly
 contained in the class of expanding automaton semigroups, and the
 class of expanding automaton semigroups is strictly contained in the class
 of asynchronous automaton semigroups.
\end{prop1}
 We show the latter by
 proving that the bicyclic monoid (the monoid with monoid
 presentation $<a,b \ | \ ab=1>$) is not a submonoid of any expanding
 automaton semigroup (Proposition \ref{left invertible iff right invertible}), and then we demonstrate an asynchronous
 automaton semigroup that contains the bicyclic monoid as a
 submonoid (Proposition \ref{bicyclic submonoid}).

 In Section 3 we investigate the dynamics of expanding automaton
 semigroups and asynchronous automaton semigroups on the trees on
 which they act.  Example \ref{Thue-Morse} gives an example of an
 expanding automaton semigroup $S$ acting on $\{0,1\}^*$ such that
 there are infinite words $\omega_1,\omega_2 \in \{0,1\}^{\omega}$ with
 $s(\omega_1)=\omega_1$ and $s(\omega_2)=(\omega_2)$ for all $s \in
 S$. Furthermore, if $\omega \in \{0,1\}^{\omega}$ is not equal to
 $\omega_1$ or $\omega_2$, then $s(\omega)\neq \omega$ for all $s
 \in S$.  Proposition \ref{aut sgp fixed points on bdy} shows that
 automaton semigroups cannot have this kind of dynamical behavior
 when acting on the boundary of a tree.  Thus the boundary dynamics of expanding
 automaton semigroups is richer than the boundary dynamics of
 automaton semigroups.

 Section 3 also investigates several algorithmic problems regarding
 the actions of expanding automaton semigroups on a tree.
 Proposition \ref{word problem solvable} gives an algorithm that
 solves the uniform word problem for expanding automaton semigroups.
  This result is already known, as Grigorchuk et al. show in
Theorem 2.15 of \cite{GVS} that the uniform word problem is solvable
for asynchronous automaton semigroups.  We give an algorithm with
our terminology for completeness.  Proposition \ref{decidability of
injectivity} gives an algorithm which decides whether a state of an
automaton over $\Sigma$ induces an injective function from
$\mathcal{T}(\Sigma^*)$ to $\mathcal{T}(\Sigma^*)$.

Since the uniform word problem is decidable for these semigroups,
there is an algorithm that takes as input an expanding automaton
over an alphabet $\Sigma$ and states $q_1,q_2$ of the automaton and
decides whether $q_1(w)=q_2(w)$ for all $w \in \Sigma^*$. On the
other hand, Theorem \ref{post corr thm} shows the following.
\begin{thm2}
\begin{enumerate}

\item There is no algorithm which takes as input an expanding
automaton
\newline $\mathcal{A}=(Q,\Sigma,t,o)$ and states $q_1,q_2 \in Q$ and
decides whether or not there is a word $w \in \Sigma^*$ with
$q_1(w)=q_2(w)$.
\item There is no algorithm which takes as input an expanding automaton
\newline $\mathcal{A}=(Q,\Sigma,t,o)$ and states $q_1,q_2 \in Q$ and decides
whether or not there is an infinite word $\omega \in
\Sigma^{\omega}$ such that $q_1(\omega)=q_2(\omega)$.
\end{enumerate}
\end{thm2}

The problem in part 1 of the above theorem is decidable for
automaton semigroups: if $\mathcal{A}=(Q,\Sigma,t,o)$ is a
synchronous automaton with $q_1,q_2 \in Q$, then (because $q_1$ and
$q_2$ induce level producing functions $\Sigma^* \rightarrow
\Sigma^*$) there is a word $w \in \Sigma^*$ such that
$q_1(w)=q_2(w)$ if and only if there is a letter $\sigma \in \Sigma$
such that $q_1(\sigma)=q_2(\sigma)$.

We close Section 3 by applying Theorem \ref{post corr thm} to study
the dynamics of asynchronous automaton semigroups.  Theorem
\ref{prefix code post corr thm} shows that there is no algorithm
which takes as input an asynchronous automaton over an alphabet
$\Sigma$, a subset $X \subseteq \Sigma$, and a state $q$ of the
automaton and decides whether there is a word $w \in X^*$ such that
$q(w)=w$.  Thus we cannot decide if $q$ has a fixed point in $X^*$.
  Furthermore, Theorem \ref{prefix code post corr thm} also shows that there is no algorithm which takes as
input an asynchronous automaton over an alphabet $\Sigma$, a subset
$X \subseteq \Sigma$, and a state $q$ of the automaton and decides
whether there is an infinite word $\omega \in X^{\omega}$ such that
$q(\omega)=\omega$.  Thus undecidability arises when trying to
understand the fixed points sets of asynchronous automaton
semigroups on the boundary of a tree.

In Section 4 we give the basic algebraic theory of expanding
automaton semigroups.  Recall that a semigroup $S$ is
\emph{residually finite} if for all $s_1,s_2 \in S$ with $s_1 \neq
s_2$ then there is a finite semigroup $S'$ and a homomorphism
$\phi:S \rightarrow S'$ such that $\phi(s_1) \neq \phi(s_2)$.
Proposition \ref{EAS res finite} shows that expanding automaton
semigroups are residually finite.  It is already known that
automaton groups are residually finite (see Proposition 2.2 of
\cite{GS}) and automaton semigroups are residually finite (see
Proposition 3.2 of \cite{C}). Asynchronous automaton semigroups are
not residually finite, as there is an asynchronous automaton
generating Thompson's group $F$ (see section 5.2 of \cite{GVS}).
This group is an infinite simple group, and so is not residually
finite.  Thus residual finiteness of expanding automaton semigroups
also distinguishes this class from the class of asynchronous
automaton semigroups.  Recall that if $S$ is a semigroup, an element
$s \in S$ is said to be \emph{periodic} if there are $m,n \in
\mathbb{N}$ such that $a^m=a^n$.  Proposition \ref{restriction on
periodicity} shows that the periodicity structure of expanding
automaton semigroups is restricted.  In particular, let $P_{\Sigma}$
denote the set of prime numbers that divide $|\Sigma|!$.  If $S$ is
an expanding automaton semigroup and $s \in S$ is such that
$s^m=s^n$ for some $m,n \in \mathbb{N}$ with $n>m$, then the prime
factorization of $n-m$ contains only primes from $P_{\Sigma}$.

In Section 4.2 we provide information about subgroups of expanding
automaton semigroups. Proposition \ref{EAS gp iff aut gp} shows that
an expanding automaton semigroup $S$ is a group if and only if $S$
is an automaton group.  Proposition 3.1 of \cite{C} shows that an
automaton semigroup $S$ is a group if and only if $S$ is an
automaton group; we use the idea of the proof of this proposition to
obtain our result. Note that such a proposition does not apply to
asynchronous automaton semigroups, as Thompson's group $F$ can be
realized with an asynchronous automaton. Proposition \ref{subgroups
of EAS's} shows that if $H$ is a subgroup of an expanding automaton
semigroup, then there is a self-similar group $G$ such that $H$ is a
subgroup of $G$. Proposition \ref{unique maximal subgroup} shows
that if an expanding automaton semigroup $S$ has a unique maximal
subgroup $H$, then $H$ is self-similar. In particular, this
proposition implies that if $S$ is the semigroup generated by the
states of an invertible synchronous automaton, then the group of
units of $S$ is self-similar (Corollary \ref{cor about inv syn
aut}).

In Section 5.1 we study closure properties of expanding automaton
semigroups. Let $S$ and $T$ be semigroups.  The \emph{normal ideal
extension} of $S$ by $T$ is the disjoint union of $S$ and $T$ with
multiplication defined by $x \cdot y=xy$ if $x,y \in S$ or $x,y \in
T$, $x \cdot y=y$ if $x \in S$ and $y \in T$, and $x \cdot y=x$ if
$x \in T$ and $y \in S$.  Note that if $S$ is a semigroup, then
adjoining a zero to $S$ is an example of a normal ideal extension.
Proposition 5.6 of \cite{C} shows that the class of automaton
semigroups is closed under normal ideal extensions.  We show in
Proposition \ref{AAS normal ideal extension} that the class of
asynchronous automaton semigroups is closed under normal ideal
extensions.  On the other hand, we show in Proposition \ref{N with
zero adjoined not EAS} that the free semigroup of rank 1 with a zero
adjoined is not an expanding automaton semigroup. Example
\ref{Thue-Morse} shows that the free semigroup of rank 1 is an
expanding automaton semigroup, and so we have that the class of
expanding automaton semigroups is not closed under normal ideal
extensions.  Lastly, we show that the class of expanding automaton
semigroups is closed under direct product (provided the direct
product is finitely generated).  In Proposition 5.5 of \cite{C} Cain
shows the same result for automaton semigroups, and our proof is
similar.

Section 5.2 contains further constructions of expanding automaton
semigroups.  A \emph{free partially commutative monoid} is a monoid
generated by a set $X=\{x_1,...,x_n\}$ with relation set $R$ such
that $R \subseteq \{(x_ix_j,x_jx_i) \ | \ 1 \leq i,j \leq n\}$, i.e.
a monoid in which the only relations are commuting relations between
generators.  We show the following.
\begin{thm1}
Every free partially commutative monoid is an automaton semigroup.
\end{thm1}

\section{DEFINITIONS AND EXAMPLES}

 Given a set $X$, let $X^+$ denote the free semigroup generated by
 $X$.  In the free monoid $X^*$, let $\emptyset$ denote the
 identity. As defined in \cite{GVS}, an {\it asynchronous automaton} is
 a quadruple $(Q,\Sigma,t,o)$ where
 $Q$ is a finite set of states, $\Sigma$ is a finite alphabet of
 symbols, $t:Q \times \Sigma \rightarrow Q$ is a transition
 function, and $o:Q \times \Sigma \rightarrow \Sigma^*$ is an output function.  A {\it synchronous
 automaton} is defined analogously, the difference being that $o:Q \times \Sigma \rightarrow
 \Sigma$ (the range of the output function is $\Sigma$ rather than $\Sigma^*$).  In
 this paper, we study a restricted class of asynchronous automata.

 An \emph{expanding automaton} is a quadruple $\mathcal{A}=(Q,\Sigma,t,o)$ where
 $Q$ is a finite set of states, $\Sigma$ is a finite alphabet of
 symbols, $t:Q \times \Sigma \rightarrow Q$ is a transition
 function, and $o:Q \times \Sigma \rightarrow \Sigma^+$ is an output
 function.  We view an expanding automaton $\mathcal{A}$ as a directed labeled graph
 with vertex set $Q$ and an edge from $q_1$ to $q_2$ labeled by
 $\sigma|w$ if and only if $t(q_1,\sigma)=q_2$ and $o(q_1,\sigma)=w$. Given an edge
 $\sigma|w$ in the graph, we refer to $\sigma$ as the \emph{input} of the edge,
 and $w$ as the \emph{output} of the edge.  The interpretation of
 this graph is that if the automaton $A$ is in state $q_1$ and reads
 symbol $\sigma$, then it changes to state $q_2$ and outputs the word
 $w$.  Thus, if we fix $q_0 \in Q$, the automaton can read a sequence
 of symbols $\sigma_1\sigma_2...\sigma_n$ and output a sequence
 $w_1w_2...w_n$ where $t(q_{i-1},\sigma_i)=q_i$ and
 $o(q_{i-1},\sigma_i)=w_i$ for $i=1,...,n$.

 Each state $q \in Q$ induces a function $\Sigma^* \rightarrow
 \Sigma^*$ in the following way: $q$ acting on $\beta$, denoted
 $q(\beta)$, is defined to be the sequence that the automaton
 outputs when the automaton starts in state $q$ and reads the sequence $\beta$.
  We also insist that $q(\emptyset)=\emptyset$.  This action of $q$ on $\Sigma^*$
  induces an action of
  $q$ on $\mathcal{T}(\Sigma^*)$.  The state $q$ induces a function
  $f_q: \mathcal{T}(\Sigma^*) \rightarrow
  \mathcal{T}(\Sigma^*)$ by $f_q(w)=q(w)$ if $w \in \Sigma^*$, and
  if $e$ is an edge in $\mathcal{T}(\Sigma^*)$ with endpoints $w$
  and $w\sigma$ then $f_q(e)=e_1e_2...e_n$ where $e_1...e_n$ is the
  unique geodesic sequence of edges in $\mathcal{T}(\Sigma^*)$ connecting $q(w)$ and
  $q(w\sigma)$.  By abuse of notation, we identify $f_q$ with $q$, as
  context should eliminate confusion.
  Considering the states of an automaton as functions leads to the
  following definition:

\begin{mydef} Given an expanding automaton $\mathcal{A}$, we say that the \emph{expanding
automaton semigroup} (respectively monoid) corresponding to
$\mathcal{A}$, denoted $S(\mathcal{A})$, is the semigroup
(respectively monoid) generated by the states of
$\mathcal{A}$.\end{mydef}

An \emph{invertible synchronous automaton} (or \emph{invertible
automaton}) is a quadruple $\mathcal{A}=(Q,\Sigma,t,o)$ where $o:Q
\times \Sigma \rightarrow \Sigma$ and, for any $q \in Q$, the
restricted function $o_q:\{q\} \times \Sigma \rightarrow \Sigma$ is
a permutation of $\Sigma$. The states of an invertible automaton
$(Q,\Sigma,t,o)$ induce bijections on $\mathcal{T}(\Sigma^*)$.
Furthermore, these functions are \emph{level-preserving}, i.e.
$|w|=|q(w)|$ for all $w \in \Sigma^*$ and $q \in Q$ (where $|\cdot|$
is the length function on $\Sigma^*$).  Thus, given an invertible
automaton $\mathcal{A}=(Q,\Sigma,t,o)$, we define the
\emph{automaton group} associated with $\mathcal{A}$ to be the group
generated by the states of $\mathcal{A}$. An \emph{automaton
semigroup} is a semigroup generated by the states of a synchronous
automaton. Thus the generators of an automaton semigroup over the
alphabet $\Sigma$ induce level-preserving functions on
$\mathcal{T}(\Sigma^*)$, but these functions are not necessarily
bijective.  Finally, an \emph{asynchronous automaton semigroup} is a
semigroup generated by the states of an asynchronous automaton.

A \emph{self-similar group} is a group generated by the states of an
invertible synchronous automaton with possibly infinitely states.  A
\emph{self-similar semigroup} is defined analogously.  Thus we
define an \emph{expanding self-similar semigroup} to be a semigroup
generated by the states of an expanding automaton with possibly
infinitely many states.

Note that if $s \in S$ where $S$ is an expanding automaton semigroup
acting on $\mathcal{T}(\Sigma^*)$, then $s$ need not induce a
level-preserving function $\mathcal{T}(\Sigma^*) \rightarrow
\mathcal{T}(\Sigma^*)$.  Thus elements of expanding automaton
semigroups are not necessarily graph morphisms.  If
$\mathcal{A}=(Q,\Sigma,t,o)$ is an expanding automaton, then the
output function mapping into $\Sigma^+$ implies that $|w|\leq
|s(w)|$ for all $s \in S(\mathcal{A})$, $w \in \Sigma^*$.  We say
that a function $f:\mathcal{T}(\Sigma^*) \rightarrow
\mathcal{T}(\Sigma^*)$ is \emph{prefix-preserving} if $f(v)$ is a
prefix of $f(w)$ in $\Sigma^*$ whenever $v$ is a prefix of $w$ in
$\Sigma^*$. We call a function $f:\mathcal{T}(\Sigma^*) \rightarrow
\mathcal{T}(\Sigma^*)$ \emph{length-expanding} if $|w|\leq |f(w)|$
for all $w \in \Sigma^*$ and $f(\emptyset)=\emptyset$.  If we
topologize the tree $\mathcal{T}(\Sigma^*)$ by making each edge
isometric to $[0,1]$ and imposing the path metric, then an element
of an expanding automaton semigroup acting on
$\mathcal{T}(\Sigma^*)$ will induce a prefix-preserving,
length-expanding endomorphism of the tree.  We call
$f:\mathcal{T}(\Sigma^*) \rightarrow \mathcal{T}(\Sigma^*)$ an
\emph{expanding endomorphism} if $f$ is prefix-preserving and
length-expanding.

Let $f$ be an expanding endomorphsm of a tree
 $\mathcal{T}(\Sigma^*)$, where $\Sigma=\{\sigma_1,...,\sigma_n\}$.  Then $f$ induces a function $\Sigma \rightarrow
 \Sigma^+$; for the rest of the paper we denote this function
 by $\tau_f$.  Note that for any $w \in \Sigma^*$, the tree
 $w\mathcal{T}(\Sigma^*)$ is isomorphic (as a graph or a metric space) to $\mathcal{T}(\Sigma^*)$.  Now for each $\sigma \in \Sigma$, $f$ induces an
 expanding endomorphism $\sigma\mathcal{T}(\Sigma^*) \rightarrow
 f(\sigma)\mathcal{T}(\Sigma^*)$; for the rest of the paper we denote this
 induced endomorphism by $f_{\sigma}$.  For any $\sigma \in \Sigma$ and $w \in
 \Sigma^*$, $f_{\sigma}$ is characterized by the equation
 \[f(\sigma w)=\tau_f(\sigma)f_{\sigma}(w).\]
The function $f_{\sigma}$ is called the \emph{section of $f$ at
$\sigma$}. Inductively, given $w \in \Sigma^*$, there exists an
expanding endomorphism $f_w$ such that $f(wv)=f(w)f_w(v)$ for every
$v \in \Sigma^*$.  We call $f_w$ the \emph{section of $f$ at $w$}.
To completely describe an expanding automorphism $f$, we need only
know the induced function $\tau_f$ and the sections
$f_{\sigma_1},...,f_{\sigma_n}$. Thus, in keeping with the notation
for automaton groups and semigroups in \cite{C} and \cite{N}, any
expanding endomorphism $f$ can be written as
\[f=(f_{\sigma_1},...,f_{\sigma_n})\tau_f\]
where each $f_{\sigma_i}$ is the section of $f$ at $\sigma_i$.

We denote a function $\tau: \Sigma \rightarrow \Sigma^+$ by
$[a_1,...,a_n]$ where $\tau(\sigma_i)=a_i$.  If $f$ and $g$ are
expanding endomorphisms with $f=(f_1,...,f_n)[w_1,...,w_n]$ and
$g=(g_1,...,g_n)[v_1,...,v_n]$, then their composition (our
functions act on the left) is given by the formula

\begin{equation}\label{composition formula}f \circ
g=(f_{v_1}g_1,...,f_{v_n}g_n)[f(v_1),...,f(v_n)]. \end{equation}

Let $\mathcal{A}=(Q,\Sigma,t,o)$ be an asynchronous automaton and $q
\in Q$.  If $w \in \Sigma^*$, then $q_w$ is obtained by viewing the
word $w$ as a path in $\mathcal{A}$ starting at $q$.  The terminal
vertex of this path is the section of $q$ at $w$.  Thus any section
of a state of $\mathcal{A}$ is itself a state of $\mathcal{A}$.

Let $\mathcal{A}=(Q,\Sigma,t,o)$ be an expanding automaton, and let
$s \in Q^*$ be an element of $S(\mathcal{A})$.  Equation
(\ref{composition formula}) allows us to build an expanding
automaton $\mathcal{A}_s$ that contains $s$ as a state.  Write
$s=q_1...q_n$. Using the original automaton $\mathcal{A}$, compute
$\tau_s$.  If we iteratively use Equation (\ref{composition
formula}) and $\mathcal{A}$, we can compute the section of $s$ at
$\sigma$ for any $\sigma \in \Sigma$ in terms of the sections of the
$q_i$'s. Furthermore, a straightforward induction on word length in
$Q^*$ shows that if $t$ is a section of $s$, then the word length of
$t$ in $Q^*$ is less than or equal to the word length of $s$ in
$Q^*$. Thus we will compute an expanding automaton $\mathcal{A}_s$
whose set of states has cardinality less than the cardinality of the
set $\{w \in Q^* \ : \ |w| \leq |s| \}$.

Before giving examples, we mention that we use the word ``action"
when describing the functions arising from these semigroups on
regular rooted trees. In general, if one says that a monoid $M$ has
an action on a set, one assumes
 that the identity of the monoid fixes each element of the set.  In this case, however, we can have expanding automaton
 monoids (and indeed automaton monoids) in which the identity element of the monoid does not
 fix each vertex of the tree, so we do not include that assumption as part of the definition of
 ``action".  Consider Example \ref{first} below.

 \begin{ex}\label{first}  Consider the expanding automaton $\mathcal{A}$ over the alphabet $\{0,1\}$ given by
 $a=(a,a)[11,1], b=(a,a)[111,11]$.  See Figure 1 for the graphical representation of $\mathcal{A}$.  We claim that $a$ is an identity
 element of $S(\mathcal{A})$ even though $a$ does not fix every
 element of $\mathcal{T}(\{0,1\}^*)$.  To see this, first note that
the range of $a$ is $\{1\}^*$.  Since $a$ fixes this set, $a^2=a$.
 Now the range of $b$ is
 $\{1\}^*-\{1\}$ and $a$ fixes this set, so $ab=b$.  Now let $w \in
 \{0,1\}^*$, and let $w_0 \in \mathbb{N}$ denote the number of 0's
 that occur in $w$; define $w_1$ analogously.  Then
 $a(w)=1^{2w_0+w_1}$, and therefore $ba(w)=1^{2w_0+w_1+1}$.
 Let $w'$ be the word obtained from $w$ by deleting the first letter
 of $w$.  If 0 is the first letter of $w$, then
 \[b(w)=1111^{2(w')_0+(w')_1}=1^{2w_0+w_1+1}=ba(w).\]
 Similarly, if $w$ starts with a 1 we have $b(w)=ba(w)$.  Hence
 $ab=b=ba$, and $a$ is an identity element.  Thus the action of $S(\mathcal{A})$ on $\{0,1\}^*$ includes the action
 of a semigroup identity that is not the identity function on $\mathcal{T}(\Sigma^*)$.
\end{ex}

\begin{figure}
\centering\includegraphics[width=55mm]{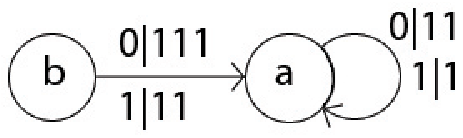}
\caption{Example \ref{first}}
\end{figure}

We now show that there are semigroups which are expanding automaton
semigroups but not automaton semigroups.

\begin{prop}\label{not aut sgps}  The class of automaton semigroups is strictly
contained in the class of expanding automaton semigroups.
\end{prop}

\begin{proof}
Let $m,n \in \mathbb{N}-\{1\}$, and let $S_{m,n}$ denote the
semigroup with semigroup presentation \newline $<a,b \ | \
b^m=b^n,ab=b>$. We show that $S_{m,n}$ is not an automaton semigroup
for any $m,n$, but $S_{m,n}$ is an expanding automaton semigroup for
any $m,n$.

Note that for any distinct $m,n \in \mathbb{N}$ with $m<n$, the
rewriting system defined by the rules $ab \rightarrow a$ and $b^n
\rightarrow b^m$ is terminating and confluent.  Thus $\{b^ja^n \ | \
j=1,...,n-1, \ n \in \mathbb{N}\}$ is a set of normal forms for
$S_{m,n}$, and so $a$ is not periodic in $S_{m,n}$.

We begin by showing $S_{m,n}$ is not an automaton semigroup.  Fix
$1<m<n$.  Let $\mathcal{A}_{m,n}=(Q,\Sigma,t,o)$ be a synchronous
automaton such that $S(\mathcal{A}_{m,n})$ is generated by two
elements $a$ and $b$ with $b^m=b^n$ and $ab=b$. We show that $a$ is
periodic in $S(\mathcal{A}_{m,n})$.  Note that $a$ and $b$ must both
be states of $\mathcal{A}_{m,n}$ as higher powers of $a$ and $b$
cannot multiply to obtain $a$, and powers of $a$ cannot multiply to
obtain $b$.  Let $\sigma_1 \in \Sigma$ be such that there exists a
minimal number $n>0$ with $a^n(\sigma_1)=\sigma_1$. Since the action
of $a$ is length-preserving, there must exist such a $\sigma_1$. Let
$\{\sigma_1,...,\sigma_{n-1}\}$ be the orbit of $\sigma_1$ under the
action of $a$ where $a(\sigma_i)=\sigma_{i+1}$ for $1\leq i \leq
n-2$ and $a(\sigma_{n-1})=\sigma_m$.

First suppose that $a_{\sigma_j}=a^{m_j}$ for each $1 \leq j \leq
n-1$.  If $m_j>1$ for some $j$, then $(a^{m_j})_{\sigma_j}=a^{n_1}$
where $n_1>m_j$, $(a^{n_1})_{\sigma_j}=a^{n_2}$ where $n_2>n_1$, and
so on.
 In this case, $a$ will have infinitely many sections, which cannot happen since $a$ is a state
 of a finite automaton.  Thus $m_j=1$ for all $j$.  Note that if $a^k(\sigma)=\sigma_1$
 for some $k>0$ and $\sigma \in \Sigma$, then the same logic implies that if $a_{\sigma}=a^r$ for some $r$ then $r=1$.  Thus we see that if
 $\sigma \in \Sigma$ and the section of $a$ at $a^k(\sigma)$ is a power of $a$ for all $k$, then the section
 of $a$ at $a^k(\sigma)$ is $a$ for all $k>0$.  Suppose that $a_{\sigma}=a$
 for all $\sigma \in \Sigma$.  Since the action of $a$ is
 length-preserving, there exist distinct $r,s \in \mathbb{N}$ such
 that $\tau_a^r=\tau_a^s$.  Then, as the only section of $a$ is $a$,
 we have $a^r=a^s$.

 Suppose now that there is a letter $\sigma \in \Sigma$ such that there exists $\sigma'$ in the
 forward orbit of $\sigma$ under the action of $a$ where
 $a_{\sigma'} \not \in <a>$.  Since $ab=b$ and $b$ is periodic, there exist distinct
 $m_{\sigma},n_{\sigma}
 \in \mathbb{N}$ with $n_{\sigma}>m_{\sigma}$ such that
 $(a^{m_{\sigma}})_{\sigma}=(a^{m_{\sigma}+k(n_{\sigma}-m_{\sigma})})_{\sigma}$ for any
 $k \in \mathbb{N}$. To see that this is true, let $t$ be the minimal number such that
 the orbit of $a^t(\sigma)$ under the
 action of $a$ is a cycle.  Since the action of $a$ is
 length-preserving, there must exist such a $t$.  Suppose that there
 is a $k \in \mathbb{N}$ such that $k\geq t$ and the section of $a$ at
 $a^k(\sigma)$ is $b^ia^j$ for some $i \in \mathbb{N}$ and
 $j \in \mathbb{N}\cup \{0\}$.
 Then the relation $ab=b$ implies that for any $k' \geq k$ we have
 $(a^{k'})_{\sigma} = b^{i'}a^j$
 for some $i'$.  Periodicity of $b$ then implies
 that there are $m_{\sigma},n_{\sigma} \geq k$ as desired.  Suppose,
 on the other hand, that the section of $a$ at $a^r(\sigma)$ is in $<a>$
 for all $r \geq t$.  Let $c$ be the maximal number such
 that the section of $a$ at $a^c(\sigma)$ is not in $<a>$ and let $p \in \mathbb{N}$.
 Then  $(a^{c+p})_{\sigma}=a^{n_p}(a^{c})_{\sigma}$ for some $n_p \in \mathbb{N}$
  and the relation $ab=b$ implies that $(a^{c+p})_{\sigma}=(a^{c})_{\sigma}$.
  In this case we let $m_{\sigma}=c$ and $n_{\sigma}=c+1$.

 Let $\hat{\Sigma}=\{\sigma \in \Sigma \ | \ (a^r)_{\sigma} \not \in <a> \ \text{for some} \ r \}$.
 By the preceding paragraph, for each $\sigma \in \hat{\Sigma}$
 choose $m_{\sigma},n_{\sigma} \in \mathbb{N}$ such that
 $(a^{m_{\sigma}})_{\sigma}=(a^{m_{\sigma}+k(n_{\sigma}-n_{\sigma})})_{\sigma}$.
 Since $a$ acts in a length-preserving
 fashion, there exist distinct $t_1,t_2$ such that $\tau_a^{t_1}=\tau_a^{t_1+k(t_2-t_1)}$
 for all $k \in \mathbb{N}$.  Thus we can choose distinct $s,t \in \mathbb{N}$ such that
 $\displaystyle\tau_{a^s}=\displaystyle\tau_{a^{s+k(t-s)}}$ and
 $(a^s)_{\sigma}=(a^{s+k(t-s)})_{\sigma}$ for all $\sigma \in \hat{\Sigma}$ and $k \in \mathbb{N}$.
 We claim that $a^s=a^t$.  To see this, let $\delta \in
 \Sigma$.  If $\eta \in \hat{\Sigma}$, then the choice of $s$ and $t$ implies that
 $(a^s)_{\eta}=(a^t)_{\eta}$.  Fix $\delta \not \in
 \hat{\Sigma}$.
 Then $(a^s)_{\delta}=a^s$ and $(a^t)_{\delta}=a^t$, so the choice of
 $s$ and $t$ implies that
 $\displaystyle\tau_{(a^s)_{\delta}}=\displaystyle\tau_{(a^t)_{\delta}}$.
 If $\eta \in \hat{\Sigma}$, then \[(a^s)_{\delta
 \eta}=(a^s)_{\eta}=(a^t)_{\eta}=(a^t)_{\delta \eta}.\]
 If $\eta \not \in \hat{\Sigma}$ then $(a^s)_{\delta \eta}=a^s$
 and $(a^t)_{\delta \eta}=a^t$, and so
 $\displaystyle\tau_{(a^s)_{\delta
 \eta}}=\displaystyle\tau_{(a^t)_{\delta\eta}}$.  Similarly, let $w \in \Sigma^*$ and write
 $w=\sigma_1...\sigma_n$.  Suppose there is an $i \in \mathbb{N}$ such that
 $\sigma_i \in \hat{\Sigma}$ and $\sigma_1,...,\sigma_{i-1} \in
 \Sigma-\hat{\Sigma}$.  Then \[(a^s)_w=(a^s)_{\sigma_i...\sigma_n}=(a^t)_{\sigma_i...\sigma_n}=(a^t)_w.\]
 On the other hand, if $w \in
 (\Sigma-\hat{\Sigma})^*$  then
 $\displaystyle\tau_{(a^s)_w}=\tau_{a^s}=\tau_{a^t}=\displaystyle\tau_{(a^t)_w}$.
 Thus $a^s=a^t$, and so $S(\mathcal{A}_{m,n})$ is not $S_{m,n}$.

Fix $1<m<n$, and let $\Sigma=\{\sigma_1,...,\sigma_n\}$ be an
alphabet.  Let $\mathcal{A}_{m,n}$ be the automaton with states $a$
and $b$ (which depend on $m,n$) defined by
\[a=(a,...,a)[\sigma_1\sigma_1,\sigma_2,...,\sigma_n], \ \ \
b=(b,...,b)\tau_b\] where
\[\tau_b(\sigma_i)=\begin{cases} \sigma_{i+1} & 1\leq i < n \\
\sigma_m & i=n \end{cases}.\]  Then $b^m=b^n$ in
$S(\mathcal{A}_{m,n})$.  Note also that the range of $b$ is
$\{\sigma_2,...,\sigma_n\}^*$, and $a$ fixes this set.  So $ab=b$.
 Now fix $i,j \in \mathbb{N}$ such that $i < n$. Then
 $b^ia^j(\sigma_1)=b^i(\sigma_1^{2^j})=\sigma_i^{2^j}$.  Thus
 $b^ia^j=b^ka^l$ in $S(\mathcal{A}_{m,n})$ if and only if $i=k$ and
 $j=l$, and we have $S(\mathcal{A}_{m,n})\cong S_{m,n}$. \end{proof}

Recall that the bicyclic monoid is the monoid with monoid
presentation $B:=<a,b \ | \ ab=1>$.  Clifford and Preston show in
Corollary 1.32 of \cite{CP} that $ba \neq 1$ in $B$.  Furthermore,
the same corollary shows that if $S$ is a semigroup and $a,b,c \in
S$ such that $c^2=c$, $ca=ac=c$, $cb=bc=b$, and $ba=c$, then the
submonoid generated by $a$,$b$, and $c$ is the bicyclic monoid if
and only if $ba \neq c$.

\begin{prop}\label{left invertible iff right invertible}
Let $S$ be an expanding automaton semigroup.  If $M$ is a submonoid
of $S$, then $M$ is not isomorphic to the bicyclic monoid.\end{prop}

\begin{proof}
Let $S$ be an expanding automaton semigroup over an alphabet
$\Sigma$.  Suppose $a,b,c \in S$ such that $c^2=c$, $ca=ac=c$,
$cb=bc=b$, and $ba=c$.  We show that $ab=c$.

If $s \in S$, let range$(s)$ denote $s(\Sigma^*)$.  Since $c$ is
idempotent, $c$ fixes range$(c)$.  The equations $ca=a$ and $cb=b$
imply that $c$ fixes range$(a)$ and range$(b)$.  Thus range$(a)$,
range$(b) \subseteq$ range$(c)$.  So we see that $a$ must act
injectively on range$(c)$: if $x,y \in $ range$(c)$ and
$a(x)=a(y)=z$, then $b(z)=x$ and $b(z)=y$ and so $x=y$. Furthermore,
because $b$ cannot reduce word length, $a$ must act in a
length-preserving fashion on range$(c)$.  Thus
range$(a)$=range$(c)$.  Now the equation $ba=c$ implies that $b$
maps range$(a)$ onto range$(c)$, and hence $b$ acts injectively and
in a length-preserving fashion on range$(a)$.  Thus if $w \in $
range$(c)$ then $bc(w)=w=a(w)$.  Suppose $w \not \in$ range$(c)$.
 Then $bab(w)=b(w)=bc(w)$.  Since $ab(w),c(w) \in$ range$(c)$ and
 $b$ acts injectively on range$(c)$, $ab(w)=c(w)$.  Thus $ab=c.$\end{proof}

We now distinguish the class of expanding automaton semigroups from
 the class of asynchronous automaton semigroups.

\begin{figure}\label{pic bicyclic submonoid}
\centering\includegraphics[width=71mm]{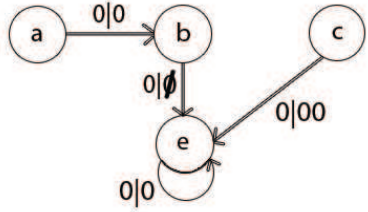}
\caption{The automaton from Proposition \ref{bicyclic submonoid}}
\end{figure}

 \begin{prop}\label{bicyclic submonoid}
The class of expanding automaton semigroups is strictly contained in
the class of asynchronous automaton semigroups.
 \end{prop}
\begin{proof}
Let $\mathcal{A}$ be the asynchronous automaton over the alphabet
$\{0\}$ with four states defined by \[ a=(b)[0], \ \
b=(e)[\emptyset], \ \ c=(e)[00], \ \ e=(e)[0].\]  Figure \ref{pic
bicyclic submonoid} gives the graphical representation of
$\mathcal{A}$.  Note that $e(0^n)=0^n$ for all $n \in \mathbb{N}$,
so $e$ is an identity element of $S(\mathcal{A})$. Note also that by
construction $ac(0^n)=0^n=e(0^n)$ for any $n$, but $ca(0)=00$. Thus
$ac=e$ but $ca \neq e$ in $S(\mathcal{A})$.  So Corollary 1.32 of
\cite{CP} implies that the submonoid generated by $a$ and $c$ is the
bicyclic monoid, and Proposition \ref{left invertible iff right
invertible} implies that $S(\mathcal{A})$ is not an expanding
automaton semigroup.
\end{proof}

\section{Decision Properties and Dynamics}

We begin this section by showing that expanding automaton semigroups
have richer boundary dynamics than automaton semigroups. Proposition
\ref{aut sgp fixed points on bdy} restricts the kind of action that
an automaton semigroup can have on the boundary of a tree, and
Example \ref{Thue-Morse} gives an expanding automaton semigroup
which shows that this restriction does not extend to the dynamics of
these semigroups.  Example \ref{Thue-Morse} also provides a
realization of the free semigroup of rank 1 as an expanding
automaton semigroup. Proposition 4.3 of \cite{C} shows that the free
semigroup of rank 1 is not an automaton semigroup, so Example
\ref{Thue-Morse} provides another example of an expanding automaton
semigroup that is not an automaton semigroup.  Let $S$ be a
semigroup acting on a set $X$, and $s \in S$.  We say that $x \in X$
is a \emph{fixed point of $s$} if $s(x)=x.$

\begin{prop}\label{aut sgp fixed points on bdy}
Let $S$ be an automaton semigroup with corresponding automaton
$\mathcal{A}=(Q,\Sigma,t,o)$. If every state of $\mathcal{A}$ has at
least two fixed points in $\Sigma^{\omega}$, then every state of
$\mathcal{A}$ has infinitely many fixed points in $\Sigma^{\omega}$.
\end{prop}
\begin{proof}
We begin with some terminology.  We call a path $p$ in $\mathcal{A}$
an \emph{inactive path} if each edge on $p$ has the form $\sigma |
\sigma$ for some $\sigma \in \Sigma$.

Let $q \in Q$.  Since $\mathcal{A}$ is a synchronous automaton, $q$
acts in a length-preserving fashion.  Since $q$ has a fixed point in
$\Sigma^{\omega}$, in the finite automaton $\mathcal{A}$ there must
exist an inactive circuit $c_1$ accessible from $q$ via an inactive
path $p$.  Let $q_1$ be a state on $c_1$. As $q_1$ must also have
two fixed points in $\Sigma^{\omega}$, either there is another
inactive circuit containing $q_1$ or there is another inactive
circuit accessible from $q_1$ via an inactive path.  In either case,
$q$ has infinitely many fixed points in $\Sigma^{\omega}$ by
``pumping" the two inactive circuits.\end{proof}

\begin{ex}\label{Thue-Morse} (Thue-Morse Automaton):  This example is constructed to model
the substitution rules which give the Thue-Morse sequence.  This
infinite binary sequence, denoted $(T_i)$, is the limit of 0 under
iterations of the substitution rules $0\rightarrow 01, 1 \rightarrow
10$. The complement of the Thue-Morse sequence, denoted
$(\overline{T_i})$, is the limit of
  $1$ under iterations of these substitution rules.  For more information on these sequences, see Section 2.2 of \cite{L} by Lothaire.

Consider the expanding automaton $\mathcal{A}$ given by
$a=(a,a)[01,10]$ over the alphabet $\Sigma = \{0,1\}$. First note
that $S(\mathcal{A})$ is the free semigroup of rank 1. To see this,
by construction of $\mathcal{A}$ we have $|a^n(0)|=2^n$ for all $n$,
and thus $a^m \neq a^n$ for any $m \neq n$.

 Also by construction of  $\mathcal{A}$, the action of $S(\mathcal{A})$ has exactly two fixed points
 in $\{0,1\}^{\omega}$: $(T_i)$ and
  $(\overline{T_i})$.  To see this, first notice that $(T_i)$ and
  $(\overline{T_i})$ are the fixed points of $a$ (see section 2.1 of
  \cite{L}).  Thus $(T_i)$ and $(\overline{T_i})$ are fixed points of
  $a^n$ for any $n$.  Furthermore, $a^n=(a^n,a^n)\tau_{a^n}$ where
  $\tau_{a^n}$ maps 0 to the prefix of length $2^n$ of $(T_i)$ and
  maps 1 to the prefix of length $2^n$ of $(\overline{T_i})$.  Thus
  section 2.1 of \cite{L} implies that $a^n$ has exactly two fixed
  points for all $n$.  \end{ex}

The following proposition gives an algorithm for solving the uniform
word problem in the class of expanding automaton semigroups.  This
proposition is a special case of Theorem 2.15 of \cite{GVS}, which
shows that asynchronous automaton semigroups have solvable uniform
word problem.  We include a proof for completeness.

\begin{prop}\label{word problem solvable}
Expanding automaton semigroups have solvable uniform word problem.
\end{prop}

\begin{proof}
Let $A=(Q,\Sigma,t,o)$ be an expanding automaton, and let $S=S(A)$.
Let $s=q_1...q_m$ and $t=q_1'...q_n'$ be elements of $S$.  If
$\tau_{s} \neq \tau_{t}$, then $s \neq t$. If $\tau_{s} = \tau_{t}$,
then use Equation (\ref{composition formula}) to calculate
$s_{\sigma}$ and $t_{\sigma}$ for all $\sigma \in \Sigma$.  If
$\tau_{s_{\sigma}} \neq \tau_{t_{\sigma}}$ for some $\sigma \in
\Sigma$, then $s \neq t$. If $\tau_{s_{\sigma}}= \tau_{t_{\sigma}} \
\forall \ \sigma \in \Sigma$, then calculate $\tau_{s_w},\tau_{t_w}$
for each $w \in \Sigma^+$ with $|w|=2$, and continue the process.
Since $|\{s_w \ : \ w \in \Sigma^*\}| \leq |\{w' \in Q^* \ : \ |w'|
\leq m \}|$ and $|\{t_w \ : \ w \in \Sigma^*\}| \leq |\{w' \in Q^* \
: \ |w'| \leq n \}|$, this process stops in finite time.\end{proof}

We now turn to showing that undecidability arises in the dynamics of
these semigroups.

\begin{thm}\label{post corr thm}
\begin{enumerate}

\item There is no algorithm which takes as input an expanding
automaton
\newline $\mathcal{A}=(Q,\Sigma,t,o)$ and states $q_1,q_2 \in Q$ and
decides whether or not there is a word $w \in \Sigma^*$ with
$q_1(w)=q_2(w)$.
\item There is no algorithm which takes as input an expanding automaton
\newline $\mathcal{A}=(Q,\Sigma,t,o)$ and states $q_1,q_2 \in Q$ and decides
whether or not there is an infinite word $\omega \in
\Sigma^{\omega}$ such that $q_1(\omega)=q_2(\omega)$.
\end{enumerate}
\end{thm}

\begin{proof}
We show undecidability by embedding the Post Correspondence Problem.
Let $X=\{x_1,...,x_m\}$ be an alphabet, and let $V=(v_1,...,v_n)$
and $W=(w_1,...,w_n)$ be two lists of words over $X$. Let
$Y=\{1,...,n\} \subseteq \mathbb{N}$ and $Z=\{z_1,z_2\}$ be
alphabets such that $X \cap Y \cap Z = \emptyset$. Undecidability of
the Post Correspondence Problem  implies that, in general, we cannot
decide if there is a sequence $(y_1,...,y_t)$ of elements of $Y$
such that $v_{y_1}v_{y_2}...v_{y_t}=u_{y_1}u_{y_2}...u_{y_t}$.

We build an expanding automaton $\mathcal{A}_{X,V,W}$ over the
alphabet $\Sigma:=X \cup Y \cup Z$ as follows.  Let the state set
$Q$ of $\mathcal{A}_{X,V,W}$ be $\{a,b\}$, and let
\[t(q,\sigma)=q \ \text{for all} \ q \in Q, \sigma \in \Sigma\]
\[o(a,i)=v_i \ \text{for} \ 1 \leq i \leq n, \ \ o(a,\sigma)=z_1 \
\text{for} \ \sigma \in \Sigma-Y\]
\[o(b,i)=w_i \ \text{for} \ 1 \leq i \leq n, \ \ o(b,\sigma)=z_2 \
\text{for} \ \sigma \in \Sigma-Y\]  Figure \ref{PCP example} shows
$\mathcal{A}_{X,U,W}$ where $X=\{s,t\}$, $V=(st,ts^2,t^2)$, and
$W=(s^2,tsts,t^2s)$.

Note that for any $w \in \Sigma^*$, $a(w)$ does not contain the
letter $z_2$; similarly, $b(w)$ does not contain the letter $z_1$.
Now if $w \in \Sigma^*$ contains a letter of $X \cup Z$, then we
know $a(w) \neq b(w)$ since $a(w)$ contains the letter $z_1$ and
$b(w)$ contains the letter $z_2$.  Thus if there is a word $w \in
\Sigma^*$ such that $a(w)=b(w)$, then $w \in Y^*$.  By construction
of $\mathcal{A}_{X,V,W}$, if $y=y_1y_2...y_n \in Y^*$ and
$a(y)=b(y)$, then
$v_{y_1}v_{y_2}...v_{y_t}=u_{y_1}u_{y_2}...u_{y_t}$.  Thus the
expanding automaton $\mathcal{A}_{X,V,W}$ simulates Post's problem,
and since we cannot decide the Post Correspondence Problem, we
cannot decide if there is a word $w \in Y^*$ with $a(w)=b(w)$. This
proves part (1).

It is shown by Rouhonen in \cite{R} that the infinite Post
Correspondence Problem is undecidable. That is, there is no
algorithm that takes as input two lists of words $v_1,...,v_n$ and
$w_1,...,w_n$ over an alphabet $X$ and decides if there is an
infinite sequence $(i_k)_{k=1}^{\infty}$ such that
$v_{i_1}v_{i_2}...=w_{i_1}w_{i_2}...$.  Thus, using the same
expanding automata and logic as above, (2) is proven. \end{proof}

\begin{figure}\label{PCP example}
\centering\includegraphics[width=72mm]{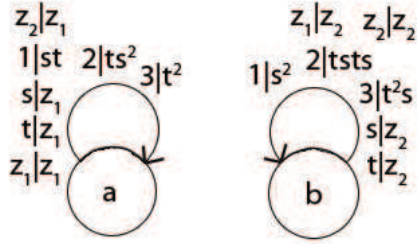} \caption{The
automaton $\mathcal{A}_{X,U,W}$ where $X=\{s,t\}$,
$V=(st,ts^2,t^2)$, and $W=(s^2,tsts,t^2s)$}
\end{figure}

We now show that undecidability arises when trying to understand the
fixed point sets of elements of asynchronous automaton semigroups.
If $w \in A^*$ for a set $A$, let $\text{Pref}_k(w)$ denote the
prefix of $w$ of length $k$.

\begin{mydef}
Let $A^*$ be a free monoid.  A subset $C \subseteq A^*$ is a {\bf
prefix code} if
\begin{enumerate}
\item C is the basis of a free submonoid of $A^*$
\item If $c \in C$, then $\text{Pref}_k(c) \not \in C$ for all $1
\leq k \leq |C|$
\end{enumerate}
\end{mydef}

The \emph{prefix code Post correspondence problem} is a stronger
form of the Post Correspondence Problem. The input of the prefix
code Post Correspondence Problem is two lists of words $v_1,...,v_n$
and $w_1,...,w_n$ over an alphabet $X$ such that $\{v_1,...,v_n\}$
and $\{w_1,...,w_n\}$ are prefix codes. A solution to the problem is
a sequence of indices $(i_k)_{1\leq k \leq m}$ with $1 \leq i_k \leq
n$ such that $v_{i_1}...v_{i_m}=w_{i_1}...w_{i_m}$. Rouhonen also
shows in \cite{R} that this form of Post's problem is undecidable.
We use the prefix code Post problem to prove the following:

\begin{thm}\label{prefix code post corr thm}
\begin{enumerate}
\item There is no algorithm that takes as input an asynchronous
automaton $\mathcal{A}$ over an alphabet $X$, a subset $Y \subseteq
X$, and a state $q$ of $\mathcal{A}$ and decides whether or not $q$
has a fixed point in $Y^*$, i.e. decides if there is a word $w \in
Y^*$ such that $q(w)=w$.
\item There is no algorithm that takes as input
an asynchronous automaton $\mathcal{A}$ over an alphabet $X$, a
subset $Y \subseteq X$, and a state $q$ of $\mathcal{A}$ and decides
whether or not $q$ has a fixed point in $Y^{\omega}$, i.e. decides
if there is an infinite word $\omega \in Y^{\omega}$ such that
$q(\omega)=\omega$.
\end{enumerate}
\end{thm}

\begin{proof}
Let $X$ be an alphabet, and let $C,D \subseteq X^*$ be prefix codes
where $C=\{c_1,...,c_m\}$ and $D=\{d_1,...,d_m\}$. Let
$\mathcal{A}_{X,C,D}$ be the expanding automaton with states $c,d$
that we constructed in the proof of Proposition \ref{post corr thm}.
Then $\mathcal{A}_{X,C,D}$ is an expanding automaton over the
alphabet $\Sigma := \{1,...,m\} \cup X \cup \{z_1,z_2\}$ such that
$o(c,i)=c_i$ and $o(d,i)=d_i$. We build an asynchronous automaton
$\mathcal{B}$ over the alphabet $\Sigma$ with a state $c'$ such that
$c'c$ is the identity function from $\{1,...,m\}^*$ to
$\{1,...,m\}^*$.  We know that there is a function $c':X^*
\rightarrow \{1,...,m\}^*$ such that $c'c$ is the identity because
$\{c_1,...,c_m\}$ generates a free monoid, so $c$ induces an
injection from $\{1,...,m\}^*$ to $X^*$.

We begin construction of $\mathcal{B}$ by starting with a single
state $c'$, and then attaching a loop based at $c'$ such that the
input letters of the loop read the word $c_1$ when read starting at
$c'$.  The corresponding output word, when read starting at $c'$, we
define to be $(\emptyset)^{|c_1|-1}1$.  In other words, the first
$|c_1|-1$ edges of the loop have the form $x | \emptyset$, and the
last edge of the loop has the form $x|1$. Next, we attach a loop at
$c'$ such that the input letters of the loop when read starting at
$c'$ read the word $c_2$, and the corresponding output word is
$(\emptyset)^{|c_2|-1}2$.  If $c_1$ ad $c_2$ have a non-trivial
common prefix, then the resulting automaton with two loops is not
deterministic.  In this case, we ``fold" the maximum length common
prefixes together, resulting in a deterministic automaton.  We
iteratively continue this process until we can read the words
$c_1,...,c_m$ as input words starting at $c'$, and $c'(c_i)=i$ for
all $i$.  Note that we can do this process since $c_i$ is not a
prefix of $c_j$ for any $i \neq j$.  At this step in the
construction of $\mathcal{B}$, $\mathcal{B}$ is a partial
asynchronous automaton, i.e. given a state of $q$ of $\mathcal{B}$,
the domain of $q$ is not all of $\Sigma^*$. However, we do have
$c'c$ is the identity function $\{1,...,m\}^* \rightarrow
\{1,...,m\}^*$ by construction of $\mathcal{B}$.  In order to make
$\mathcal{B}$ an asynchronous automaton, we add a sink state $i$
such that $t(i,\sigma)=i$ and $o(i,\sigma)=\sigma$ for all $\sigma
\in \Sigma$.  Now for each $q$ a state in $\mathcal{B}$ and $\sigma
\in \Sigma$ such that there there is no edge out of $q$ with
$\sigma$ as an input letter, define $t(q,\sigma)=i$ and
$o(q,\sigma)=z_1$.

Recall that in the proof of Theorem \ref{post corr thm}, in general
we cannot find $w \in \{1,...,m\}^*$ such that $c(w)=d(w)$ because
such a $w$ is a solution to the Post Correspondence Problem.  By
construction of $\mathcal{B}$, any $w \in \{1,...,m\}^*$ such that
$c'd(w)=w=c'c(w)$ is a solution to the prefix code Post
Correspondence Problem. Now $c'd$ is an element of the asynchronous
automaton semigroup generated by the states of $\mathcal{A}_{X,C,D}$
and $\mathcal{B}$. Thus, undecidability of the prefix code Post
Correspondence Problem implies part (1).

In \cite{R}, Ruohonen shows that the that there is no algorithm
which takes as input two lists of words $v_1,...,v_n$ and
$w_1,...,w_n$ over an alphabet $X$ such that $\{v_1,...,v_n\}$ and
$\{w_1,...,w_n\}$ are prefix codes and decides whether there is an
infinite sequence of indices $(i_k)_{k=1}^{\infty}$ such that
$v_{i_1}v_{i_2}...=w_{i_1}w_{i_2}...$.  Thus the same logic and
automata as above implies part (2).\end{proof}

We now give an algorithm which determines whether or not an element
of an expanding automaton semigroup induces an injection
$\mathcal{T}(\Sigma^*) \rightarrow \mathcal{T}(\Sigma^*)$.  Before
we do this, we must recall some basic automata theory which can be
found in more detail in Chapters 1 and 2 of \cite{HU} by Hopcroft
and Ullman. In order to avoid ambiguity of language in this paper,
we will use the phrase ``deterministic finite state automaton" to
denote a 5-tuple $(Q,\Sigma,\delta,q_0,F)$ where $Q$ is a state set,
$\Sigma$ is an alphabet, $\delta$ is a partial function from $Q
\times \Sigma$ to $Q$, $q_0$ is an initial state, and $F$ is a set
of final states. Let ``nondeterministic finite state automaton"
denote a 5-tuple $(Q,\Sigma,\delta,q_0,F)$ where $Q$ is a state set,
$\Sigma$ is an alphabet, $\delta$ is a partial relation from $Q
\times \Sigma$ to $Q$ that is not a partial function, $q_0$ is an
initial state, and $F$ is a set of final states. We view a finite
state automaton (deterministic or nondeterministic) as a finite
directed graph with vertex set $Q$ and an arrow from $q_1$ to $q_2$
labeled by $\sigma$ if and only if $\delta(q_1,\sigma)=q_2$. Given a
finite state automaton $M=(Q,\Sigma,\delta,q_0,F)$, call a directed
edge path $p$ an \emph{acceptable path} in $M$ if $p$ begins at
$q_0$ and ends at a final state. If $M$ is a deterministic or a
nondeterministic finite state automaton, let $\phi_M:
\{\text{acceptable paths in M}\} \rightarrow L(M)$ (where $L(M)$
denotes the language accepted by $M$) denote the map which sends a
path $p$ to the word in $L(M)$ that labels the path $p$.  If $M$ is
deterministic then $\phi_M$ is injective.  We show in the following
lemma that we can decide if $\phi_M$ is injective for a
nondeterministic finite state automaton $M$.

\begin{lem}\label{injectivity lemma}
Let $M$ be a nondeterministic finite state automaton.  Then there is
algorithm that decides whether or not $\phi_M$ is injective.
\end{lem}

\begin{proof}
Let $M=(Q,\Sigma,\delta,q_o,F)$ be a nondeterministic finite state
automaton.  We build a deterministic finite state automaton
$M'=(Q',\Sigma,\delta',q_0',F')$ from $M$ using a construction of
Hopcroft and Ullman from chapter 1 of \cite{HU} as follows.  The
state set $Q'$ is the power set of $Q$, $q_0'=\{q_0\}$, and $F'=\{S
\in Q' \ | \ \text{there exists} \ q \in S \ \text{such that} \ q
\in F\}$.  Lastly,
$\delta'(\{q_1,...,q_k\},\sigma)=\{\delta(q_1,\sigma),...,\delta(q_k,\sigma)\}$.
Then, by construction of $M'$, $\phi_M$ is not injective if and only
if in $M'$ there is an edge $\{r_1,...,r_t\} \xrightarrow{\sigma}
\{s_1,...,s_v\}$ accessible from $\{q_0\}$ such that there exist
distinct $r_{i_1},r_{i_2}$ with $\delta(r_{i_j},\sigma) \in F$ or
there is an $r_j$ such that in $M$ there are two edges coming out of
$r_j$ labeled by $\sigma$ whose terminal vertices are final states.
\end{proof}

\begin{prop}\label{decidability of injectivity}
Let $\mathcal{A}=(Q,\Sigma,t,o)$ be an expanding automaton.  Given
$q \in Q$, there is an algorithm to decide if $q: \Sigma^*
\rightarrow \Sigma^*$ is injective.\end{prop}
\begin{proof}
Fix $q \in Q$.  First we build a finite state automaton
$M=(Q',\Sigma,\delta,q_0,F)$ from $\mathcal{A}$.  Begin with state
set $Q'$ in bijection with $Q$.  Whenever $q_1 \xrightarrow{\sigma |
w} q_2$ in $\mathcal{A}$ with $w=v_1...v_k$ where $v_i \in \Sigma$,
add enough states in $M$ so that there is a path labeled by
$v_1...v_k$ from $q_1'$ to $q_2'$. Intuitively, $M$ is the finite
state automaton we get from $\mathcal{A}$ by dropping the inputs off
of the edges in $\mathcal{A}$, then making each edge into a path so
that every edge in $M$ is labeled by an element of $\Sigma$.  Let
$F=Q'$ and $q_0=q'$.  Note that $q$ is not injective if and only if
there exist distinct paths in $\mathcal{A}$ such that the outputs
read along each path give the same element of $\Sigma^*$.  Now $M$
is constructed so that for each $w \in \text{range}(q)$ there exists
an acceptable path $p$ in $M$ such that $\phi_{M}(p)=w$, and given
an acceptable path $p'$ in $M$ we have $\phi_{M}(p') \in
\text{range}(q)$. Furthermore, each acceptable path in $M$
corresponds to an input path in $\mathcal{A}$.  Thus $q$ is not
injective if and only if there exist two distinct paths $p_1$ and
$p_2$ in $M$ such that $\phi_{M}(p_1)=\phi_{M}(p_2)$. By lemma
\ref{injectivity lemma}, there is an algorithm to decide this
property of $M$.\end{proof}

The set of semigroups that can be realized by expanding automata
such that the states induce injective functions is very restricted.
 Let $S$ be an expanding automaton semigroup with corresponding
automaton $\mathcal{A}=(Q,\Sigma,t,o)$ such that each state $q \in
Q$ induces an injection $\mathcal{T}(\Sigma^*) \rightarrow
\mathcal{T}(\Sigma^*)$.  Then any element of $Q^*$ also induces an
injection $\mathcal{T}(\Sigma^*) \rightarrow \mathcal{T}(\Sigma^*)$.
Let $e \in S$, and suppose that $e$ is idempotent.  Since $e$ is
idempotent, $e$ fixes range$(e)$. If $w \in \Sigma^*$ is such that
$e(w)\neq w$, then $w$ and $e(w)$ are both preimages of $e(w)$ under
$e$.  Since $e$ induces an injection, we have that $e$ is the
identity function on $\mathcal{T}(\Sigma^*)$.  Let $e_{\Sigma}$
denote the identity function on $\mathcal{T}(\Sigma^*)$. Then $S$
can contain at most one idempotent, namely $e_{\Sigma}$.  If
$e_{\Sigma} \in S$, then Proposition \ref{unique maximal subgroup}
implies that the group of units of $S$ is self-similar.  Suppose
that $e_{\Sigma} \not \in S$. Then $S$ contains no idempotents and
hence any $s \in S$ is non-periodic.

Suppose that there is an $s \in S$ such that there exists a word $w
\in \Sigma^*$ with $|w| < |s(w)|$.  Then, because each element of
$S$ is injective and elements of $S$ cannot shorten word length when
acting on $\Sigma^*$, there cannot be an element $s' \in S$ such
that $ss's=s$.  A semigroup $T$ is said to be \emph{von Neumann
regular} if for each $t \in T$ there is a $t' \in T$ with $tt't=t$.
Then $S$ is not von Neumann regular.  Thus we have shown the
following.

\begin{prop}
Let $S$ be an expanding automaton semigroup over an expanding
automaton $\mathcal{A}=(Q,\Sigma,t,o)$ such that each $q$ induces an
injective function $\mathcal{T}(\Sigma^*) \rightarrow
\mathcal{T}(\Sigma^*)$. Then \begin{enumerate}
\item The group of units of $S$ is self-similar.
\item $S$ is von Neumann regular if and only if $\mathcal{A}$ is an invertible automaton and $S$ is a group.
\item If $e \in S$ is idempotent then $e = e_{\Sigma}$.
\item If $e_{\Sigma} \not \in S$, then $S$ does not contain any
periodic elements.
\end{enumerate}
\end{prop}

\section{Algebraic Properties}

\subsection{Residual Finiteness and Periodicity}
In this section we show that expanding automaton semigroups are
residually finite and that the periodicity structure of these
semigroups is restricted.

\begin{prop}\label{EAS res finite}
Expanding automaton semigroups are residually finite.\end{prop}
\begin{proof}
Let $S$ be an expanding automaton semigroup over the alphabet
$\Sigma$ and let $a,b \in S$ with $a \neq b$. For each $m \in
\mathbb{N}$, let $L(m)=\{w \in \Sigma^* \ : \ |w|=m\}$, i.e. $L(m)$
is the $m$th level of the tree $\Sigma^*$. Since $a$ and $b$ are
distinct, there is $n \in \mathbb{N}$ such that $a$ and $b$ act
differently on $L(n)$.  Let
\[n'=\text{max}\{|a(w)|,|b(w)| \ : \ w \in L(n)\} \] and let
$\mathcal{L}=\left(\bigcup_{i=1}^{n'}L(i)\right) \cup \{\$\}$.
Finally, let $T(\mathcal{L})$ denote the semigroup of
transformations $\mathcal{L} \rightarrow \mathcal{L}$.  Since
$\mathcal{L}$ is finite, $T(\mathcal{L})$ is a finite semigroup.
 Define a homomorphism $\rho:S \rightarrow T(\mathcal{L})$ by
$\rho(s)=f$
where $f(\$)=\$$ and \[f(x)=\begin{cases} s(x) & s(x) \in \mathcal{L} \\
\$ & s(x) \not \in \mathcal{L} \end{cases} \] Since $a$ and $b$ act
differently on $L(n)$, construction of $\rho$ ensures that $\rho(a)$
and $\rho(b)$ are distinct in $T(\mathcal{L})$.\end{proof}

Let $G$ be an automaton group over an alphabet $\Sigma$ and let
$P_{\Sigma}$ denote the set of prime numbers that divide
$|\Sigma|!$. If $g \in G$ has finite order, then the order of $g$
must have only primes from $P_{\Sigma}$ in its prime factorization.
One can see this by considering $g$ as a level-preserving
automorphism on a tree of degree $|\Sigma|$, and thus the
cardinality of any orbit under the action of $g$ must have only
prime numbers dividing $|\Sigma|!$ in its prime factorization.  We
show an analogous proposition for the periodicity structure of
expanding automaton semigroups.  First, we define a \emph{partial
invertible automaton} to be a quadruple $(Q,\Sigma,t,o)$ where $t$
is a partial function from $Q \times \Sigma$ to $Q$ and $o$ is a
partial function from $Q \times \Sigma$ to $\Sigma$ such that the
restricted partial function $o_q$  from $\{q\} \times \Sigma$ to
$\Sigma$ is a partial permutation of $\Sigma$.  It is
straightforward to show that any partial invertible automaton can be
``completed" to an invertible automaton, i.e. given a partial
invertible automaton $\mathcal{B}$ there is an invertible automaton
$\mathcal{A}$ (not necessarily unique) such that $\mathcal{B}$
embeds (via a labeled graph homomorphism) into $\mathcal{A}$.

\begin{prop}\label{restriction on periodicity}
Let $S$ be an expanding automaton semigroup over an alphabet
$\Sigma$, and let $P_{\Sigma}$ be as above.  If $s \in S$ is
periodic with $s^m=s^n$, $m<n$, and $s,...,s^{n-1}$ distinct, then
$n-m$ has only primes from $P_{\Sigma}$ in its prime factorization.
\end{prop}
\begin{proof}
Let $\mathcal{A}=(Q,\Sigma,t,o)$ be an expanding automaton with
$S=S(\mathcal{A})$.  Suppose $s \in S$ is periodic with $s^m=s^n$,
$m<n$, and $s,...,s^{n-1}$ distinct. Fix $w \in s^m(\Sigma^*)$. Then
$R_w:=\{s^k(w) \ | \ k \geq m\}$ is a finite set, and the
cardinality of $R_w$ divides $n-m$.  Note that for any $w' \in R_w$,
$s$ acts like a cycle on $w'$ as $s^m(w')=s^n(w')$.  Furthermore, if
$v,v' \in R_w$ then $|v|=|v'|$ because $s$ cannot shorten word
length.  Thus the paths in $\mathcal{A}$ corresponding to the input
words $s^m(w),...,s^{n-1}(w)$ form a partial invertible subautomaton
of $\mathcal{A}$.  Denote this partial invertible subautomaton by
$\beta_w$.  Consider the partial invertible subautomaton  $\beta$ of
$\mathcal{A}$ given by $\beta=\displaystyle\cup_{w \in \Sigma^*}
(\beta_w)$.  Complete $\beta$ to an invertible automaton $\beta'$.
Then $R_w$ is an orbit under the action of an element of an
automaton group for all $w \in \Sigma^*$, and the result
follows.\end{proof}

\subsection{Subgroups}

Let $\mathcal{A}=(Q,\Sigma,t,o)$ be a synchronous invertible
automaton.  As in \cite{N}, construct an automaton
$\mathcal{A}^{-1}=(Q^{-1},\Sigma,t^{-1},o^{-1})$ where $Q^{-1}$ is a
set in bijection with $Q$ under the mapping $q \rightarrow q^{-1}$,
$t^{-1}(q_1^{-1},\sigma)=q_2^{-1}$ if and only if
$t(q_1,\sigma)=q_2$, and $o(q^{-1},\sigma_1)=\sigma_2$ if and only
if $o(q,\sigma_2)=\sigma_1$.  Then $qq^{-1}=q^{-1}q$ is the identity
function $\Sigma^* \rightarrow \Sigma^*$, and so $\mathcal{A}^{-1}$
is called the \emph{inverse automaton for $\mathcal{A}$} (we always
denote the inverse automaton for $\mathcal{A}$ by
$\mathcal{A}^{-1}$).

\begin{prop}\label{EAS gp iff aut gp}

A group $G$ is an automaton group (respectively self-similar group)
if and only if $G$ is an expanding automaton semigroup (respectively
expanding self-similar semigroup).\end{prop}

\begin{proof}
Let $G$ be an automaton group corresponding to the automaton
$\mathcal{A}:= (Q,\Sigma,t,o)$.  Since $G$ is an automaton group,
$\mathcal{A}$ is invertible and synchronous. Construct a new
automaton $\mathcal{B}=\mathcal{A} \cup \mathcal{A}^{-1}$.  Then
$S(\mathcal{B})=G$ and $\mathcal{B}$ is an expanding automaton. Thus
$G$ is an expanding automaton semigroup.

Conversely, let the group $G$ be an expanding automaton semigroup
corresponding to the expanding automaton $\mathcal{A}=
(Q,\Sigma,t,o)$.  Let $e$ be the identity of $G$ and $g \in G$. Then
\[e(\Sigma^*)=g(g^{-1}(\Sigma^*)) \subseteq g(\Sigma^*)\] and
\[g(\Sigma^*)=e(g(\Sigma^*)) \subseteq e(\Sigma^*) \] Hence
$e(\Sigma^*)=g(\Sigma^*)$.  Now $e$ is idempotent and thus fixes
$e(\Sigma^*)$, so (as in the proof of \ref{left invertible iff right
invertible}) $g$ is bijective and length-preserving on
$g(\Sigma^*)=e(\Sigma^*)$. Thus $G$ is isomorphic to the semigroup
generated by $\{g|_{e(\Sigma^*)} \ : \ g \in G \}$.

Construct an invertible automaton $\mathcal{B}=(\overline{Q}\cup
\{1\},\Sigma=\{\sigma_1,...,\sigma_n\},\overline{t},\overline{o})$
as follows. The state set is $\overline{Q} \cup \{i\}$ where
$\overline{Q}$ is a set in bijection with $Q$ and
$i=(i,...,i)[\sigma_1,...,\sigma_n]$, i.e. $i$ is a sink state that
pointwise fixes $\Sigma^*$. The transition function is given by
\[\overline{t(q,\sigma)}=\begin{cases} t(q,\sigma) & \ \text{if} \
\sigma \in e(\Sigma) \\ i & \ \text{if} \ \sigma \not \in e(\Sigma)
\end{cases} \] and the output function is given by
\[\overline{o(q,\sigma)}=\begin{cases} o(q,\sigma) & \ \text{if} \
\sigma \in e(\Sigma) \\ \sigma & \ \text{if} \ \sigma \not \in
e(\Sigma)
\end{cases} \]

Let $g \in G$ and let $w \in \Sigma^*-e(\Sigma^*)$ be of minimal
length. Write $w=v\sigma$ where $v \in e(\Sigma^*)$.  Then the above
conditions imply that, for any $w' \in \Sigma^*$,
$\hat{q}(ww')=q(w)\sigma w'$.  In other words, each state
$\overline{q}$ of $\mathcal{B}$ will mimic the action of $q$ on
words that are in the image of $e$, but will enter the state $i$ and
act identically on the suffix of a word $w$ following the largest
prefix of $w$ lying in $e(\Sigma^*)$. So the part of the action
which does not act bijectively and in a length-preserving fashion
collapses to the identity, and we have an invertible automaton
giving $G$.

None of the above uses that the automata have finitely many states,
so the same logic shows that $G$ is a self-similar group if and only
if $G$ is an expanding self-similar semigroup.  \end{proof}

The idea in the last proof allows us to prove the following:
\begin{prop}\label{subgroups of EAS's}
Let $S$ be an expanding automaton semigroup and $H$ a subgroup of
$S$. Then there is a self-similar group $G$ with $H \leq
G$.\end{prop}

\begin{proof}
Let $S$ be an expanding automaton semigroup and $H$ a subgroup of
$S$.  Let $e$ denote the identity of $H$. Let
$\mathcal{A}=(Q,\Sigma,t,o)$ be the expanding automaton associated
with $S$. As in the proof of Proposition \ref{EAS gp iff aut gp},
$H$ is isomorphic to the semigroup generated by $\{h|_{e(\Sigma^*)}
\ : \ h \in H \}$ and each element of $H$ acts injectively and in
length-preserving fashion on $e(\Sigma^*)$.  Then we can again
collapse the ``non-group" part of the action to the state which
fixes the tree to get a length-preserving and invertible action of
$H$. Thus we can construct an invertible (and possibly infinite
state) synchronous automaton containing the elements of $H$ as
states.  The states generated by this automaton is a self-similar
group $G$ with $H \leq G$.\end{proof}

If $S$ is an expanding automaton semigroup and $H$ is a subgroup of
$S$, then $S$ is a subgroup of a self-similar group, but $H$ is not
necessarily self-similar.  If $H$ is the unique maximal subgroup of
$S$, then we show below that $H$ is self-similar.

\begin{prop}\label{unique maximal subgroup}
Let $S$ be an expanding automaton semigroup with a unique maximal
subgroup $G$. Then $G$ is self-similar.\end{prop}

\begin{proof}
Let $\mathcal{A}=(Q,\Sigma,t,o)$ be the automaton associated with
$S$.  Let $g \in G$ and write $g=(g_1,...g_n)\tau_g$ where
$n=|\Sigma|$.  Let $e$ be the identity of $G$, and write
$e=(e_1,...,e_n)\tau_e$.  Since $e$ is idempotent, $e$ fixes
range$(e)$, and thus the set $\hat{\Sigma}:=\{\sigma \in \Sigma \ |
\ e(\sigma)=\sigma \}$ is non-empty.  To see this, let $\sigma \in
\Sigma$ and suppose that $e(\sigma)=\sigma'w$ for some $\sigma' \in
\Sigma$.  Then $e$ fixes $\sigma'w$, and since $e$ is
length-expanding $e(\sigma')=\sigma'$.  Since $e$ is idempotent,
$e_{\hat{\sigma}}$ is idempotent for all $\hat{\sigma} \in
\hat{\Sigma}$.  This is true because
$(e^n)_{\hat{\sigma}}=(e_{\hat{\sigma}})^n$.  Since $G$ is the
unique maximal subgroup of $S$, there is only one idempotent element
of $S$.  Thus $e_{\hat{\sigma}}=e$ for all $\hat{\sigma} \in
\hat{\Sigma}$.

Let $\sigma \in \hat{\Sigma}$.  Then $\tau_g(\sigma) \in
\hat{\Sigma}$ and so $e_{\tau_g(\sigma)}=e$.  Thus Equation
(\ref{composition formula}) implies
\[ g_{\sigma}=(eg)_{\sigma}=e_{\tau_g(\sigma)}g_{\sigma}\]
and, as $e$ stabilizes $\sigma$,
\[g_{\sigma}=(ge)_{\sigma}=g_{\sigma}e_{\sigma}.\]
Hence $eg_{\sigma}=g_{\sigma}e=g_{\sigma}$ for any $\sigma \in
\hat{\Sigma}$.

Let $h=g^{-1}$, $\sigma \in \hat{\Sigma}$, and write
$h=(h_1,...,h_n)\tau_h$. By the same logic as above,
$eh_{\sigma}=h_{\sigma}e=h_{\sigma}$.  Since $hg=e$ we have
\[(hg)_{\sigma}=h_{\tau_g(\sigma)}g_{\sigma}=e_{\sigma}=e\]
Since $g_{\sigma}$ is left-invertible, Proposition \ref{left
invertible iff right invertible} implies that $g_{\sigma}$ is
invertible. Therefore $g_{\sigma} \in G$ for all $\sigma \in
\hat{\Sigma}$.

Continuing inductively, we see that $g_w \in G$ for all $w \in
\text{range}(e)$.  Similar to the proof of Proposition \ref{EAS gp
iff aut gp}, if $w \not \in  \text{range}(e)$ then for all $g \in G$
we can replace $g_w$ with $e$ and the resulting group will still be
isomorphic to $G$.  This is because, as in Proposition \ref{EAS gp
iff aut gp}, the action of $G$ on range$(e)$ captures all of the
group information.  Thus $G$ is an expanding self-similar semigroup,
and Proposition \ref{EAS gp iff aut gp} implies that $G$ is a
self-similar group.\end{proof}

If $\mathcal{A}$ is an invertible synchronous automaton, then
$S(\mathcal{A})$ has at most one idempotent, namely the identity
function on the tree.  Thus Proposition \ref{unique maximal
subgroup} has the following corollary.

\begin{cor}\label{cor about inv syn aut}
Let $\mathcal{A}$ be an invertible synchronous automaton.  Then the
group of units of $S(\mathcal{A})$ is self-similar.
\end{cor}

\section{Closure Properties and further examples}

\subsection{Closure Properties}
We begin this section by showing that the class of expanding
automaton semigroups is not closed under taking normal ideal
extensions.  In particular, we show that the free semigroup of rank
1 with a zero adjoined is not an expanding automaton semigroup.  We
then show that the class of asynchronous automaton semigroups is
closed under taking normal ideal extensions.

\begin{lem}\label{free sgp rank 1 with 0 not aut sgp}
The free semigroup of rank 1 with a zero adjoined is not an
automaton semigroup.\end{lem}

\begin{proof}
Let $S$ be an automaton semigroup over an alphabet
$\Sigma=\{\sigma_1,...,\sigma_n\}$ such that $S$ is generated by two
elements $a$ and $b$ with $ab=ba=b$ and $b^2=b$.  We use the same
idea of the proof of Proposition \ref{not aut sgps} to show that $a$
is periodic.

Let $\sigma \in \Sigma$.  Suppose that the section of $a$ at
$a^n(\sigma)=b$ for some $n$.  Then
$(a^n)_{\sigma}=(a^{n+k})_{\sigma}$ for all $k \in \mathbb{N}$. If
the section of $a$ at $a^n(\sigma)$ is a power of $a$ for all $n$,
then (as in the proof of Proposition \ref{not aut sgps}) the section
of $a$ at $a^n(\sigma)$ is $a$ for all $n$.

Let $\hat{\Sigma}=\{\sigma \in \Sigma \ | \ (a^r)_{\sigma}=b \
\text{for some} \ r \}$.  As in the previous proof, we can choose
$s$ and $t$ such that $\tau_{a^s}=\tau_{a^t}$ and
$(a^s)_{\sigma}=(a^t)_{\sigma}$ for all $\sigma \in \hat{\Sigma}$.
Then the same logic of the previous proof shows that
$a^s=a^t$.\end{proof}

We now apply Lemma \ref{free sgp rank 1 with 0 not aut sgp} to show
the following.

\begin{prop}\label{N with zero adjoined not EAS}
The class of expanding automaton semigroups is not closed under
taking normal ideal extensions.  In particular, the free semigroup
of rank 1 with a zero adjoined is not an expanding automaton
semigroup.\end{prop}
\begin{proof}
Let $S=<a,b \ | \ b^2=b, \ ab=ba=b>$ be the free semigroup of rank 1
with a zero adjoined, and suppose $S$ were an expanding automaton
semigroup corresponding to the automaton
$\mathcal{A}=(Q,\Sigma,t,o)$. Since $b$ is idempotent, $b$ fixes
range$(b)$. Hence the set $\hat{\Sigma}=\{\sigma \in \Sigma \ | \
b(\sigma)=\sigma\}$ is non-empty.  Since $b$ is the only idempotent
of $S$, $b_{\hat{\sigma}}=b$ for all $\hat{\sigma} \in
\hat{\Sigma}$.

Let $\sigma \in \Sigma-\hat{\Sigma}$, and suppose that
$b_{\sigma}=a^n$ for some $n>0$.  Let $w \in \Sigma^*$.  Then
$b(\sigma w)=b(\sigma)a^n(w)$.  Since $b$ fixes range$(b)$, we have
that $b(b(\sigma w))=b(\sigma)a^n(w)$.  We also have that $b$ fixes
$b(\sigma)$ and the section of $b$ at $b(\sigma)$ is $b$. Thus $b$
fixes $a^n(w)$, and (as $w$ is arbitrary) $ba^n=a^n$ in $S$.  But
$ba^n=b$, which implies that $a^n$ is idempotent.  Since $a^n$ is
not idempotent in $S$, we have $b_{\sigma}=b$ for all $\sigma \in
\Sigma$.  Note that $b$ must be a state of $\mathcal{A}$ as powers
of $a$ cannot multiply to obtain $b$.  Thus, in the graphical
representation of $\mathcal{A}$, all edges going out of $b$ are
loops based at $b$.  Note also that $a$ must be a state of
$\mathcal{A}$.

Let $\Gamma = \{\sigma \in \Sigma \ : \ |a^m(\sigma)|=1 \ \text{for
all} \ m\}$.  The equation $ab=b$ implies that $a$ fixes range$(b)$,
and so $\Gamma$ is nonempty.  In $\mathcal{A}$, for each state $q$
in $<a>$ and $\gamma \in \Gamma$ there is an arrow labeled by
$\gamma | \hat{\gamma}$ coming out of $q$ where $\hat{\gamma} \in
\Gamma$. Let $w \in \Gamma^*$ with $w=\gamma_1...\gamma_k$.  Suppose
that $|a(w)|>1$. Then $w$, as a path in $\mathcal{A}$ based at $a$,
must enter the state $b$. Choose $i$ maximal so that
$\gamma_1...\gamma_{i-1}$ is a path such that the initial vertex of
each edge is not the state $b$. Then $a(w)=\gamma_1'...\gamma_k'$
where $\gamma_m' \in \Gamma$ for $1 \leq m \leq i-1$ and $\gamma_m'
\in \hat{\Sigma}^*$ for $i \leq m \leq k$.  Since $a$ fixes
$\hat{\Sigma}^*$, $|a^n(w)|=|a^2(w)|$ for all $n \geq 2$.  Thus for
any $w \in \Gamma^*$, $|a^{|\Sigma|}(w)|=|a^k(w)|$ for any $k \geq
|\Sigma|$.

Suppose that $|a(\sigma)|=1$ for all $\sigma \in \Sigma$.  Then the
same logic as in the proof of Proposition \ref{not aut sgps} shows
that either $a$ is periodic or has infinitely many sections (note
that the proof does not use that the periodic element acts in a
length-preserving fashion).  So the set $\Sigma'=\{\sigma \in \Sigma
\ : \ |a(\sigma)|>1 \}$ is nonempty.  Let $\sigma' \in \Sigma'$, and
write $a(\sigma')=\sigma_1...\sigma_m$ where $\sigma_i \in \Sigma$.
Suppose that $\sigma_i=\sigma'$ for some $i$.  Then
$b(a(\sigma'))=b(\sigma_1...\sigma_n)=b(\sigma_1)...b(\sigma_m)=b(\sigma')$,
and so $|b(a(\sigma'))|>|b(\sigma')|$, a contradiction. Thus
$\sigma'$ is not a letter of $a(\sigma')$.  The same calculation
also shows that $\sigma'$ is not a letter of $a^n(\sigma')$ for any
$n$ and that $\sigma'$ is not a letter of $a(\sigma_i)$ for any $i$.

 Let $w \in \Sigma^*$ and write
$w=\sigma_1...\sigma_k$. Suppose that $\sigma_i \not\in \Gamma$ for
some $i$. Then every edge in $\mathcal{A}$ with input label
$\sigma_i$ has an output label without $\sigma_i$ as a letter. Thus
$a^n(w)$ does not contain $\sigma_i$ as a letter for any $n$.  If
$a(w) \in \Gamma^*$, then as mentioned above $a$ will act in a
length-preserving fashion on $a^{|\Sigma|}(w)$.  Suppose that $a(w)
\not \in \Gamma^*$ where $\sigma_j \not \in \Gamma$ is a letter of
$a(w)$. Then $a^2(w)$ does not contain $\sigma_i$ or $\sigma_j$ as a
letter. Continuing inductively, we see that $a^{|\Sigma|}(w) \in
\Gamma^*$. Thus there is an $m \in \mathbb{N}$ such that $a$ acts in
a length-preserving fashion on $a^m(w)$ for any $w \in \Sigma^*$,
i.e. $|a^m(w)|=|a^k(w)|$ for $k \geq m$ and any $w \in \Sigma^*$.
This induces a length-preserving action of $S$ on $\Gamma^*$,
contradicting Lemma \ref{free sgp rank 1 with 0 not aut
sgp}.\end{proof}

\begin{prop}\label{AAS normal ideal extension}
Let $S$ and $T$ be asynchronous automaton semigroups.  Then the
normal ideal extension of $S$ by $T$ is an asynchronous automaton
semigroup.\end{prop}

\begin{proof}
Let $\mathcal{A}=(Q_1,\Sigma,t_1,o_1)$ and
$\mathcal{B}=(Q_2,\Gamma,t_2,o_2)$ be asynchronous automata with
$S(\mathcal{A})=S$ and $S(\mathcal{B})=T$.  Construct a new
automaton $\mathcal{C}=(Q_1 \cup Q_2, \Sigma \cup \Gamma,t,o)$ with
transition and output functions as follows:
\[t(q_1,\sigma)=t_1(q_1,\sigma) \ \text{for all} \  q_1 \in Q_1 \ \text{and} \ \sigma \in
\Sigma\]
\[t(q_1,\gamma)=q_1 \ \text{for all} \ q_1 \in Q_1 \ \text{and} \ \gamma \in \Gamma\]
\[t(q_2,\sigma)=q_2 \ \text{for all} \ q_2 \in Q_2 \ \text{and} \ \sigma \in \Sigma \]
\[t(q_2,\gamma)=t_2(q_2,\gamma) \ \text{for all} \ q_2 \in Q_2 \ \text{and} \ \gamma \in \Gamma \]
\[o(q_1,\sigma)=o_1(q_1,\sigma) \ \text{for all} \ q_1 \in Q_1 \ \text{and} \ \sigma \in
\Sigma\]
\[o(q_1,\gamma)=\gamma \ \text{for all} \ q_1 \in Q_1 \ \text{and} \ \gamma \in \Gamma\]
\[o(q_2,\sigma)=\emptyset \ \text{for all} \ q_2 \in Q_2 \ \text{and} \ \sigma \in \Sigma \]
\[o(q_2,\gamma)=o_2(q_2,\gamma) \ \text{for all} \ q_2 \in Q_2 \ \text{and} \ \gamma \in \Gamma
\]  By construction of $\mathcal{C}$, the subsemigroup of
$S(\mathcal{C})$ generated by $Q_1$ is $S$ and the subsemigroup of
$S(\mathcal{C})$ generated by $Q_2$ is $T$.

Now let $w \in (\Sigma \cup \Gamma)^*$. Write
$w=\sigma_1\gamma_1\sigma_2\gamma_2...\sigma_n\gamma_n$ with
$\sigma_i \in \Sigma^*$ and $\gamma_j \in \Gamma^*$.  Let $s \in
Q_1^*$ and $t \in Q_2^*$.  Then
\[ts(w)=t(s(\sigma_1)\gamma_1s(\sigma_2)\gamma_2...s(\sigma_n)\gamma_n)=\emptyset t(\gamma_1)\emptyset t(\gamma_2)...t(\gamma_n)=t(\gamma_1)t(\gamma_2)...t(\gamma_n)\]
and
\[st(w)=s(t(\gamma_1)t(\gamma_2)...t(\gamma_n))=t(\gamma_1)t(\gamma_2)...t(\gamma_n)\]
Thus both $st(w)$ and $ts(w)$ equal $t(w)$, so $st=ts=t$.\end{proof}

We close this section by showing that the class of expanding
automaton semigroups is closed under direct product, provided the
direct product is finitely generated.

\begin{prop}\label{closure under direct product}
Let $S$ and $T$ be expanding automaton semigroups.  Then $S \times
T$ is an expanding automaton semigroup if and only if $S \times T$
is finitely generated.
\end{prop}

\begin{proof}
An expanding automaton semigroup must be finitely generated, so the
forward direction is clear. Suppose that $S \times T$ is finitely
generated.  Then $S \times T$ is generated by $A \times B$ for some
finite $A \subseteq S$ and $B \subseteq T$.  Let $\mathcal{A}_S$ and
$\mathcal{A}_T$ be expanding automata with state sets $P$ and $Q$
respectively such that $S=S(\mathcal{A}_S)$ and
$T=S(\mathcal{A}_T)$.  Furthermore, choose $m,n$ so that $A
\subseteq P^m$ and $B \subseteq Q^n$, and add enough states to each
expanding automaton so that we obtain new automata $\mathcal{A}'_S$
and $\mathcal{A}'_T$ with $S=S(\mathcal{A}'_S)$,
$T=S(\mathcal{A}'_T)$, and $P^m$ is contained in the state set of
$\mathcal{A}'_S$; likewise for $Q^n$ and $\mathcal{A}'_T$. Write
$\mathcal{A}'_S=(X',C,t',o')$ and
$\mathcal{A}'_T=(Y',D,\hat{t},\hat{o}).$ with $C$ and $D$ disjoint.

Let $\mathcal{Y}=(X' \cup Y', C \cup D,t,o)$ be the expanding
automaton defined by \[t(q,\sigma)= \begin{cases} t'(q,\sigma) & q
\in X' \ \text{and} \ \sigma \in C \\ q & q \in X' \ \text{and} \
\sigma \in D \\ \hat{t}(q,\sigma) & q \in Y' \ \text{and} \ \sigma
\in D
\\ q & q\in Y' \ \text{and} \ \sigma \in C \end{cases} \ \ \ \
 \text{and} \ \ \ \ o(q,\sigma)= \begin{cases} o'(q,\sigma) & q \in X' \
\text{and} \ \sigma \in C \\ \sigma & q \in X' \ \text{and} \ \sigma
\in D \\ \hat{o}(q,\sigma) & q \in Y' \ \text{and} \ \sigma \in D
\\ \sigma & q\in Y' \ \text{and} \ \sigma \in C \end{cases}.\]
Then the subsemigroup of $S(\mathcal{Y})$ generated by $X'$ is $S$
and the subsemigroup of $S(\mathcal{Y})$ generated by $Y'$ is $T$,
and construction of $\mathcal{Y}$ implies that $x'y'=y'x'$ for all
$x' \in X'$ and $y' \in Y'$.  Thus \newline $S(\mathcal{Y}) \cong S
\times T$.\end{proof}

Let $S$ and $T$ be finitely generated semigroups such that $T$ is
infinite. Robertson, Ruskuc, and Wiegold show in \cite{RRW} that if
$S$ is finite then $S \times T$ is finitely generated if and only if
$S^2=S$.  If $S$ is infinite, then $S\times T$ is finitely generated
if and only if $S^2=S$ and $T^2=T$.  Let $\mathbb{N}$ denote the
free semigroup of rank 1.  Then $\mathbb{N}^2\neq \mathbb{N}$, and
thus $\mathbb{N} \times \mathbb{N}$ is not an expanding automaton
semigroup (even though $\mathbb{N}$ is an expanding automaton
semigroup).

\subsection{Free Partially Commutative Monoids}

In this section we show that any free partially commutative monoid
is an expanding automaton semigroup. Let $M$ be a free partially
commutative monoid with monoid presentation $<X|R>$.  We begin by
defining the \emph{shortlex normal form} on $M$.  First, if $v \in
X^*$, $|v|$ will always denote the length of $v$ in $X^*$.  Order
the set $X$ by $x_{i}<x_j$ whenever $i<j$.  If $v,w \in X^*$, let
$v<w$ if and only if $|v|<|w|$ or, if $|v|=|w|$, $v$ comes before
$w$ in the dictionary order induced by the order on $X$.  This is
called the \emph{shortlex ordering} on $X^*$.  To obtain the set of
shortlex normal forms of $M$, for each $w \in M$ choose a word $w'
\in X^*$ such that $w=w'$ in $M$ and $w'$ is minimal in $X^*$ with
respect to the shortlex ordering.  We remark that it is immediate
from this definition that a word $w \in X^*$ is in shortlex normal
form in $M$ if and only if for all factorizations $x=ybuaz$ in $M$
where $y,u,z \in X^*$, $a$ and $b$ commute, and $a<b$, there is a
letter of $u$ which does not commute with $a$.

For any $w \in X^*$, let $w(x_i,x_j)$ denote the word obtained from
$w$ by erasing all letters except $x_i$ and $x_j$.  We write
$w(x_i)$ to denote the word obtained from $w$ by deleting all
occurrences of the letter $x_i$.  We will need the following lemma
regarding free partially commutative monoids.

\begin{lem}\label{lemma regarding trace monoids}
Let $M$ be a free partially commutative monoid generated by
$X=\{x_1,...,x_n\}$, and let $v,w \in X^*$ such that $v$ and $w$ are
in shortlex normal form in $M$. Suppose that
\begin{enumerate}
\item $|v(x_i)|=|w(x_i)| \ \text{for} \ 1
\leq i \leq n$ and
\item $v(x_i,x_j)=w(x_i,x_j) \ \text{in} \ X^* \ \text{whenever} \ 1 \leq
i,j \leq n \ \text{and} \ x_i \ \text{and} \ x_j \ \text{do not
commute}$.
\end{enumerate}  Then $v=w$ in $M$.
\end{lem}

\begin{proof}
Let $v,w \in M$ be words satisfying $|v(x_i)|=|w(x_i)|$ for all $i$.
This implies that the number of occurrences of $x_i$ as a letter of
$v$ equals the number of occurrences of $x_i$ as a letter of $w$.
In particular, $|v|=|w|$. Write $v=x_{i_1}...x_{i_k}$ and
$w=x_{j_1}...x_{j_k}$ with $v,w$ in shortlex normal form. Suppose
that $x_{i_1}<x_{j_1}$. Then $v(x_{i_1},x_{j_1}) \neq
w(x_{i_1},x_{j_1})$ in $X^*$, and condition (2) in the hypotheses
implies that $x_{i_1}$ and $x_{j_1}$ commute. Condition (1) implies
that $x_{j_1}$ is a letter of $v$ and $x_{i_1}$ is a letter of $w$,
and so we write $v=x_{i_1}v_1x_{j_1}v_2$ where $v_1$ does not
contain $x_{j_1}$ as a letter.  Similarly, write
$w=x_{j_1}w_1x_{i_1}w_2$. Condition (2) implies that $x_{i_1}$
commutes with every letter of $w_1$.  Since $x_{i_1}<x_{j_1}$, we
have that $w$ was not in lexicographic normal form. Thus $x_{i_1}
\not < x_{j_1}$, and symmetry implies $x_{j_1} \not < x_{i_1}$.  So
$x_{i_1}=x_{j_1}$. Inductively continuing the argument implies that
$x_{i_t}=x_{j_t} \ \forall \ 1 \leq t \leq k$.\end{proof}

\begin{thm}\label{trace monoids aut sgps}
Every free partially commutative monoid is an automaton semigroup.
\end{thm}
\begin{proof}
Let $M$ be a partially commutative monoid generated by
$X=\{x_1,...,x_n\}$.  Let $N=\{\{i,j\} \ | \ x_i \ \text{and} \ x_j
\ \text{do not commute} \}$.  Let $A=\{a_1,...,a_n\}$,
$B=\{b_1,...,b_n\}$, $C=\{c_{ij} \ | \ i<j \ \text {and} \ \{i,j\}
\in N \}$, and $D=\{d_{ij} \ | \ i<j \ \text{and} \ \{i,j\} \in N
\}$ be four alphabets where $C,D$ are in bijective correspondence
with $N$. We construct an automaton $\mathcal{A}_M$ with state set
$Q:=\{y_1,...,y_n,1\}$ over the alphabet $\Sigma = A \cup B \cup C
\cup D$ such that $S(\mathcal{A}_M) \cong M$ as follows.  Let $1$ be
the sink state that pointwise fixes $\Sigma^*$.  For each $i$,
define
\[ t(y_i,a_j)= 1 \ \text{for all} \ j, \ \ t(y_i,b_j)= \begin{cases} y_i & i=j \\ 1 & i \neq j \end{cases}
\] and
\[o(y_i,a_j)= \begin{cases} b_j & i=j \\ a_j & i \neq j \end{cases},
\ \ o(y_i,b_j)= \begin{cases} a_j & i=j \\ b_j & i \neq j
\end{cases}.  \]

By construction, the subautomaton consisting of the states $y_i$ and
$1$ over the alphabet $\{a_i,b_i\}$ is the \emph{adding machine}
automaton (see Figure 1.3 of \cite{N}) for all $i$.  Note that for
any $k>j$, $y_i^j(a_i^{2j})\neq y_i^k(a^{2j})$, and so the semigroup
corresponding to this subautomaton is the free monoid of rank 1 for
all $i$.  Thus each $y_i$ acts non-periodically on $\{a_i,b_i\}^*$
for all $i$. Furthermore, if $i \neq j$ then $y_j$ induces the
identity function from $x\Sigma^*$ to $x\Sigma^*$ where $x \in
\{a_i,b_i\}$.

We now complete the construction of $\mathcal{A}$.  Fix $i<j$ with
$\{i,j\} \in N$, and let $k \in \mathbb{N}$ such $1\ \leq k \leq n$
and $k \neq i,j$.  Define

\[t(y_i,c_{ij})=y_j,  \ \ t(y_i,d_{ij})=y_i, \ \ t(y_j,c_{ij})=y_i, \  \ t(y_j,d_{ij})=y_j\]
\[o(y_i,c_{ij})=d_{ij},  \ \ o(y_i,d_{ij})=c_{ij}, \ \ o(y_j,c_{ij})=c_{ij},  \ \ o(y_j,d_{ij})=d_{ij}\]
\[t(y_k,c_{ij})=t(y_k, d_{ij})=1  \]
\[o(y_k,c_{ij})=c_{ij}, \ \ t(y_k, d_{ij})=d_{ij}.\]  For all other $i',j'$ such that
$\{i',j'\} \subseteq N$ and $i'<j'$, define the output and
transition function analogously.  Figure \ref{FPCM example} gives
the automaton $\mathcal{A}_M$ where $M$ is the free partially
commutative monoid
\newline $<y_1,y_2,y_3 \ | \ y_1y_2=y_2y_1, \ y_1y_3=y_3y_1>$ (we
omit the arrow on the sink state).

\begin{figure}\label{FPCM example}
\centering\includegraphics[width=85mm]{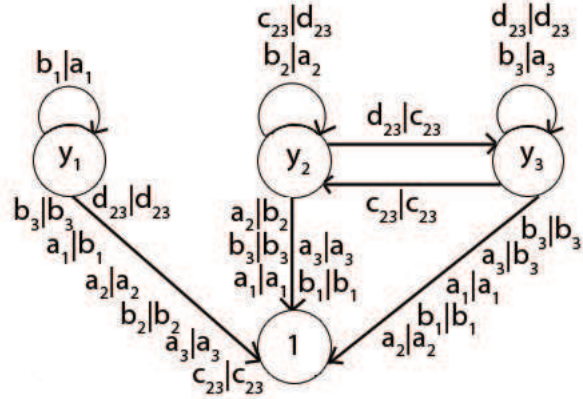}
\caption{An automaton generating the monoid $<y_1,y_2,y_3 \ | \
y_1y_2=y_2y_1, \ y_1y_3=y_3y_1>$}
\end{figure}

For each $\{i,j\} \in N$, the subautomaton of $\mathcal{A}_M$
corresponding to the states $y_i$ and $y_j$ over the alphabet
$\{c_{ij},d_{ij}\}$ is the ``lamplighter automaton" (see Figure 1.1
of \cite{GZ}).  Grigorchuk and Zuk show in Theorem 2 of \cite{GZ}
that this automaton generates the lamplighter group, and in
particular in Lemma 6 of \cite{GZ} they show that the states of this
automaton generate a free semigroup of rank 2.  Thus $y_i$ and $y_j$
generate a free semigroup of rank 2 when acting on
$\{c_{ij},d_{ij}\}^*$, and hence the semigroup generated by $y_i$
and $y_j$ in $S(\mathcal{A}_M)$ is free of rank 2.

Let $1\leq i,j \leq n$ be such that $\{i,j\} \not\subseteq N$.  By
construction of $\mathcal{A}_M$, $y_i$ and $y_j$ have disjoint
support, i.e. the sets $\{w \in \Sigma^* \ | \ y_i(w) \neq w\}$ and
$\{w \in \Sigma^* \ | \ y_j(w) \neq w\}$ are disjoint. Thus if $x_i$
and $x_j$ commute in $M$, then $y_i$ and $y_j$ commute in
$S(\mathcal{A}_M)$.   So $S(\mathcal{A}_M)$ is a quotient of $M$.

Let $v,w \in Q^*$ such that $v$ and $w$ are written in shortlex
normal form when considered as elements of $M$.  Suppose that
$w(y_i) \neq v(y_i)$ for some $i$. By construction of
$\mathcal{A}_M$, for any $i \neq j$ we have $y_j$ acts as the
identity function on $\{a_i,b_i\}^*$. Thus the action of $v$ and $w$
on $\{a_i,b_i\}^*$ is the same as the action of $v(y_i)$ and
$w(y_i)$, respectively, on $\{a_i,b_i\}^*$. So $w(y_i) \neq v(y_i)$
implies that $v \neq w$ in $S(\mathcal{A}_M)$. Hence $v=w$ in
$S(\mathcal{A}_M)$ implies that $w(y_i)=v(y_i)$ for all $i$.

Suppose now that there exist $\{r,s\} \in N$ such that $v(y_r,y_s)
\neq w(y_r,y_s)$.  If $t \neq r,s$, then $y_t$ acts like the
identity function on $\{c_{rs},d_{rs}\}^*$. Thus the action of $v$
and $w$ on $\{c_{rs},d_{rs}\}^*$ is the same as the action of
$v(y_r,y_s)$ and $w(y_r,y_s)$, respectively, on
$\{c_{rs},d_{rs}\}^*$. So $v(y_r,y_s) \neq w(y_r,y_s)$ implies that
$v\neq w$ in $S(\mathcal{A_M})$.  Thus if $v=w$ in
$S(\mathcal{A_M})$ then $v(y_r,y_s) = w(y_r,y_s)$ in $Q^*$ for all
$\{r,s\} \in N$.

The last two paragraphs have shown that if $v=w$ in
$S(\mathcal{A_M})$, then $v$ and $w$ satisfy the hypotheses of Lemma
\ref{lemma regarding trace monoids}.  Hence $v=w$ in $M$, and the
result follows.\end{proof}

\section*{Acknowledgements}
The author would like to thank his advisors Susan Hermiller and John
Meakin for their direction and guidance in the editing of this
paper.

\end{document}